\documentclass{article}
\usepackage[utf8]{inputenc}
\usepackage{amsmath, amssymb, amsthm}
\usepackage{graphicx}
\usepackage{xcolor}
\usepackage{bbm}
\usepackage{geometry}
\usepackage{color} %Color definition
\usepackage{hyperref} %Internal and external links
\usepackage{graphicx} %Graphics inclusion
\usepackage{pstricks}
\usepackage{amsmath,amssymb,amsthm,xcolor}
\usepackage{mathabx}
\usepackage{bbm}
\usepackage{tikz}
\usepackage{bm}
\usepackage{float}
\usepackage{scalerel}
\usepackage[title]{appendix}
\usepackage{makecell}

\newtheorem{lemma}{Lemma}[section]
\newtheorem{proposition}{Proposition}[section]
\newtheorem{definition}{Definition}[section]
\newtheorem{theorem}{Theorem}
\newtheorem{corollary}[lemma]{Corollary}
\newtheorem{hypothesis}{Hypothesis}

\theoremstyle{remark}
\newtheorem{remark}[theorem]{Remark}

\newcommand{\N}{\mathbb{N}}
\newcommand{\R}{\mathbb{R}}
\newcommand{\Z}{\mathbb{Z}}
\newcommand{\C}{\mathbb{C}}

\newcommand{\E}{\mathbb{E}}
\newcommand{\Var}{\mathbb{V}\text{ar}}
\newcommand{\Cov}{\C\text{ov}}
\newcommand{\ov}[1]{\overline{#1}}
\newcommand{\p}{\partial}

\newcommand{\ind}{\mathbbm{1}}

\renewcommand{\t}{\widetilde}

\newcommand\floor[1]{\lfloor#1\rfloor}
\DeclareMathOperator{\supp}{supp}
\renewcommand{\centerdot}{\abullet}
\newcommand\abullet{{\scaleobj{0.6}{\bullet}}}
\renewcommand{\P}{\mathbb{P}}
\renewcommand{\deg}{\mathrm{deg}}

\title{Concentration for integrable directed polymer models}
\author{Christian Noack\thanks{Department of Mathematics, Cornell University} \and Philippe Sosoe\thanks{Department of Mathematics, Cornell University} }
\date{}

\begin{document}
\maketitle

\begin{abstract}
    In this paper, we consider four integrable models of directed polymers for which the free energy is known to exhibit KPZ fluctuations. A common framework for the analysis of these models was introduced in \cite{CN2018}. 
    
    We derive estimates for the central moments of the partition function, of any order, on the near-optimal scale $N^{1/3+\epsilon}$, using the iterative method we applied to the semi-discrete polymer in \cite{NS}.
    Among the innovations exploiting the invariant structure, we develop formulas for correlations between functions of the free energy and the boundary weights that replace the Gaussian integration by parts appearing in our previous paper \cite{NS}. 
\end{abstract}

\section{Introduction}
In this paper, we consider four models for 1+1 dimensional integrable polymers in random environment, and study the higher moments of the centered free energy: the log-gamma polymer, introduced by Sepp\"al\"ainen \cite{S}; the strict-weak polymer, which was simultaneously introduced and analyzed by Corwin-Sepp\"al\"ainen-Shen \cite{CSS} and O'Connell-Ortmann \cite{OO2015}; the beta random walk of Barraquand and Corwin \cite{BC}; and the inverse beta model introduced by Thiery and Le Doussal \cite{TD}.  

The models in question are distinguished because they each possess algebraic structure that has enabled the verification of several predictions regarding their fluctuations. These include upper and lower bounds for the variance of the free energy, of order $O(N^{2/3})$ (see \cite{S} for log-gamma and \cite{CN} for the three other models) as well as asymptotic Tracy-Widom distribution (see \cite{Borodin} for the log-gamma polymer, the original papers \cite{CSS, OO2015,BC} for the strict-weak polymer and beta random walk models, as well a formal argument for the inverse beta model in \cite{TD}).
Results of this type are characteristic of the KPZ universality class \cite{C}, and are expected to hold in a more general setting where the integrable structure is not available, but proving this is out of reach using current methods. 

We note that the techniques used to prove asymptotic Tracy-Widom distribution by relating the free energy to a Fredholm determinant are markedly different from those that have been used to obtain variance bounds starting with the work of Sepp\"al\"ainen \cite{S}. The ideas in that work have their origins in earlier work of Sepp\"al\"ainen and co-authors on fluctuations of one-dimensional interacting particle systems \cite{BS, BSK, BSQ, BCS}. Despite their power, the Bethe ansatz methods used to obtain the asymptotic distribution are not easily adapted to estimating the size of the central moments.

Here, we build on our previous paper on the O'Connell-Yor polymer \cite{OY}, a semi-discrete 1+1 dimensional polymer model, to obtain bounds of nearly optimal order for all the central moments of the free energy in the stationary version of each of the four models mentioned above. Our main result, Theorem \ref{theorem: Main 1}, states that for each $k \ge 1$, the $k$th central moment of the free energy in a system of size $O(N^2)$ is bounded by $O(N^{k/3+\epsilon})$, where the implicit constant depends on $\epsilon$. 

The proof in \cite{NS} proceeded by deriving a pair of inequalities which appear related to the physicists' \emph{KPZ scaling relations} and which enable an iterative proof of the bound for the order of fluctuations of the free energy by successive improvements starting from the trivial $O(N^{1/2})$ bound. A crucial idea was the repeated application of Gaussian integration by parts to relate cross-terms involving the "boundary Brownian motion" component of the free energy and the free energy itself to quenched cumulants of the first vertical jump of the polymers paths. This tool is not available in the discrete models we consider here. Nevertheless, we develop a substitute for it by introducing a sequence of polynomials which play a role analogous to that of Hermite polynomials for the O'Connell-Yor polymer, and allow us to derive formulas for the cumulants of the partition function in terms of quenched cumulants of the time of the first jump. Here, the Mellin transform framework introduced in \cite{CN2018} plays a central role. See Section \ref{sec: app-ibp}.

\subsection{The polymer model}
% We consider a class of 1+1 dimensional directed polymers on the integer lattice. 

To each edge $e$ of the $\Z_+^2$  lattice we assign a positive random weight. The superscripts 1 and 2 are used to denote horizontal and vertical edge weights, respectively. For $z\in \N^2$, let $Y^1_z$ and $Y^2_z$ denote the horizontal and vertical incoming edge weights, see Figure \ref{fig-weights}. We assume that the collection of pairs $\{(Y^1_z,Y^2_z)\}_{z\in \N^2}$ is independent and identically distributed with common distribution $(Y^1,Y^2)$, but do not insist that $Y^1_z$ is independent of $Y^2_z$. $\{(Y^1_z,Y^2_z)\}_{z\in \N^2}$ are the \emph{bulk weights}. For $x\in \N\times \{0\}$, let $R^1_x$ denote the horizontal incoming edge weight, and for $y\in \{0\}\times \N$, let $R^2_y$ denote the vertical incoming edge weight. We take the collections $\{R^1_x\}_{x\in \N\times \{0\}}$ and $ \{R^2_y\}_{y\in \{0\}\times \N}$ to independent and identically distributed, with common distributions $R^1$ and $R^2$. We refer to these as the \emph{horizontal} and \emph{vertical} \emph{boundary weights}, respectively. We further assume that the horizontal boundary weights, the vertical boundary weights, and the bulk weights are independent of each other. This assignment of edge weights is illustrated in Figure \ref{fig-weights}. We call 
\begin{equation}\label{environment}
\omega = \{R^1_x, R^2_y, (Y^1_z,Y^2_z): x\in \N\times \{0\}, y\in \{0\}\times \N, z\in \N^2\}
\end{equation}
the \emph{polymer environment}. We use $\P$ and $\E$ to denote the probability measure and corresponding expectation of the polymer environment.

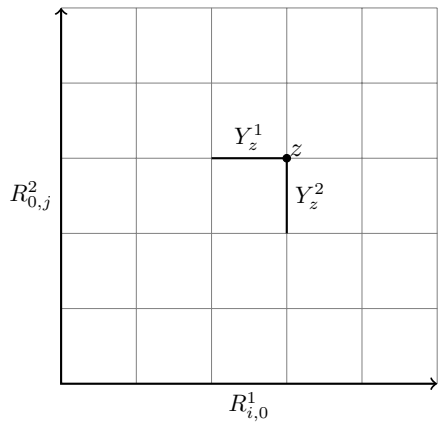
\begin{figure}[ht]
  \centering
  \begin{tikzpicture}
    
   \draw [help lines] (0,0) grid (5,5);	  
   \draw [thick, <->] (0,5) -- (0,0) -- (5,0);
   \draw [thick] (2,3) -- (3,3) -- (3,2);
   \node [left,scale=.9] at (0,2.5) {$R^2_{0,j}$};
   \node [below,scale=.9] at (2.5,0) {$R^1_{i,0}$};
   \node [above,scale=.9] at (2.5,3) {$Y^1_z$};
   \node [right,scale=.9] at (3,2.5) {$Y^2_z$};
   \draw[fill] (3,3) circle [radius=.05];
   \node [above right] at (2.91,2.91) {$z$};

  \end{tikzpicture}
    \caption{Assignment of edge weights.}
  \label{fig-weights}
\end{figure}

The weight of a path is given by the product of the weights along its edges. For $(m,n)\in
\Z_+^2\setminus \{(0,0)\}$ we define a probability measure on all up-right paths from
$(0,0)$ to $(m,n)$. See Figure \ref{up-right path} for an example of an up-right path. Let $\Pi_{m,n}$ denote the collection of all such
paths. We identify paths $x_\abullet = (x_0, x_1, \ldots, x_{m+n})$ either
their sequence of vertices or their sequence
of edges $(e_1, \ldots, e_{m+n})$, where $e_i = \{x_{i-1}, x_i\}$, as convenient. Define the quenched polymer measure on $\Pi_{m,n}$,
\[ 
Q_{m,n} (x_\abullet) := \frac{1}{Z_{m,n}}
\prod_{i=1}^{m+n} \omega_{e_i}, 
\] 
where $\omega_e$ is the weight associated to the edge $e$ and
\[ 
Z_{m,n} := \sum_{x_\abullet \in \Pi_{m,n}} \prod_{i=1}^{m+n}
\omega_{e_i} 
\] 
is the associated partition function. At the origin, define $Z_{0,0}:=1$. Taking the expectation $\E$ of the quenched measure with respect to the edge weights gives the annealed measure on $\Pi_{m,n}$,
\begin{equation}
P_{m,n}(x_\centerdot) :=\E[Q_{m,n}(x_\centerdot)].\label{def: annealed prob}
\end{equation}
The annealed expectation will be denoted by $E_{m,n}$.

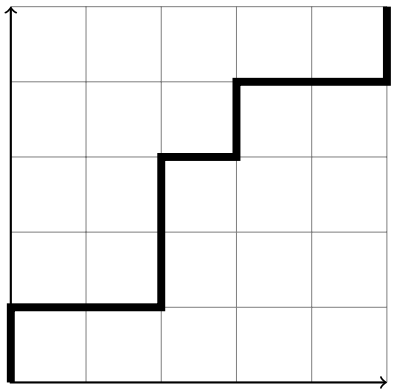
\begin{figure}[ht]
\centering
	\begin{tikzpicture}
  		\draw [help lines] (0,0) grid (5,5);	  
   		\draw [thick, <->] (0,5) -- (0,0) -- (5,0);
   		\draw [line width=3pt] (0,0) -- (0,1) -- (1,1) -- (2,1) -- (2,2) -- (2,3) -- (3,3) -- (3,4) --(4,4) -- (5,4) -- (5,5);
	\end{tikzpicture}
	\caption{An up-right path from $(0,0)$ to $(5,5)$.}
    \label{up-right path}
\end{figure}
We specify the edge weight distributions for the four stationary polymer models.
The notation $X\sim \text{Ga}(\alpha,\beta)$ is used to denote that a random variable is gamma$(\alpha,\beta)$ distributed, i.e.\ has density 
\[\frac{1}{\Gamma(\alpha)} \beta^\alpha x^{\alpha-1}e^{-\beta x}\] supported on $(0,\infty)$, where $\Gamma(\alpha) = \int_0^\infty x^{\alpha-1} e^{-x} dx$ is the gamma function. $X\sim \text{Be}(\alpha,\beta)$ is used to say that $X$ is beta$(\alpha,\beta)$ distributed, i.e.\ has density \[\frac{\Gamma(\alpha+\beta)}{\Gamma(\alpha) \Gamma(\beta)} x^{\alpha-1}(1-x)^{\beta-1},\] supported on $(0,1)$. We then use $X\sim \text{Ga}^{-1}(\alpha,\beta)$ and $X\sim \text{Be}^{-1}(\alpha,\beta)$ to denote that $X^{-1}\sim \text{Ga}(\alpha,\beta)$ and $X^{-1}\sim \text{Be}(\alpha,\beta)$, respectively. We also use $X\sim \left(\text{Be}^{-1}(\alpha,\beta) -1\right)$ to denote that $X+1\sim\text{Be}^{-1}(\alpha,\beta)$.

Each of the four models we consider is obtained by choosing the distribution of the boundary and bulk weights according to one of the four following specifications:
\begin{itemize}
\item \textbf{Inverse-gamma (IG)}: This is also known as the log-gamma model. Assume $\mu>\theta>0, \, \beta>0$ and
\begin{equation}
\begin{gathered}
R^1 \sim \text{Ga}^{-1}(\mu-\theta,\beta)\qquad  R^2\sim \text{Ga}^{-1}(\theta,\beta)\\
(Y^1,Y^2)= (X,X) \qquad \text{where} \qquad X\sim \text{Ga}^{-1}(\mu,\beta).
\end{gathered}\label{model-IG}
\end{equation}
\item \textbf{Gamma (G)}: This is also known as the strict-weak model. Assume $\theta,\,\mu,\,\beta>0$ and 
\begin{equation}
\begin{gathered}
R^1\sim \text{Ga}(\mu+\theta,\beta)\qquad R^2\sim \text{Be}^{-1}(\theta,\mu)\\
(Y^1,Y^2) = (X, 1) \qquad \text{where} \qquad X\sim \text{Ga}(\mu,\beta).
\end{gathered}\label{model-G}
\end{equation}
\item \textbf{Beta (B)}: 
Assume $\theta,\,\mu,\,\beta>0$ and 
\begin{equation}
\begin{gathered}
R^1\sim \text{Be}(\mu+\theta,\beta)\qquad R^2 \sim \text{Be}^{-1}(\theta,\mu)\\
(Y^1,Y^2) = (X,1-X) \qquad \text{where} \qquad X\sim \text{Be}(\mu,\beta).
\end{gathered}\label{model-B}
\end{equation}
\item \textbf{Inverse-beta (IB)}:
Assume $\mu>\theta>0,\, \beta>0$ and 
\begin{equation}
\begin{gathered}
R^1 \sim \text{Be}^{-1}(\mu-\theta,\beta)\qquad R^2 \sim \left(\text{Be}^{-1}(\theta,\beta+\mu-\theta)-1\right)\\
(Y^1,Y^2) = (X, X-1) \qquad \text{where} \qquad X\sim \text{Be}^{-1}(\mu,\beta).
\end{gathered}\label{model-IB}
\end{equation}
\end{itemize}
Note that each of these four choices in fact generates a family of models by choosing different values of the parameters $\mu, \theta,\beta$.
The name of each model refers to the distribution of the bulk weights. We call these models the \textbf{four basic beta-gamma models}.

\subsection{Main Result}
Having defined the models we will consider, we are now ready to state our main result. 
Given a path $x_\centerdot\in \Pi_{m,n}$, define the \emph{exit points} of the path from the horizontal and vertical axes by
\begin{align}
t_1 := \max\{i:(i,0)\in x_\centerdot\} \qquad \text{and} \qquad t_2:= \max\{j:(0,j)\in x_\centerdot\}.\label{def exit points}
\end{align}

\begin{theorem}\label{theorem: Main 1}
Assume that the polymer environment has edge weight distributions $R^1, R^2, (Y^1,Y^2)$ as in one of \eqref{model-IG} through \eqref{model-IB}, and let $(m,n)=(m_N,n_N)_{N=1}^\infty$ be a sequence such that
\begin{align}
|m_N-N \Var[\log R^2]|\leq \gamma N^{2/3} \qquad \text{and} \qquad |n_N-N\Var[\log R^1]|\leq \gamma N^{2/3} \label{direction}
\end{align}
for some fixed $\gamma>0$.  Then for every $\epsilon >0$ and $p>0$, there exists a constant $C=C(\epsilon,p)>0$ such that for any $n\in \N$,
\begin{align}
    \E[|\ov{\log Z_{m,n}}|^p]&\leq C N^{\frac{1}{3}p+\epsilon}\quad \text{ and }\label{eqn: Main theorem 1}\\
    E_{m,n}[(t_j)^p]&\leq C N^{\frac{2}{3}p+\epsilon}\quad \text{ for both } j=1,2.\label{eqn: Main theorem 2}
\end{align}

\end{theorem}

We also obtain exact formulas for the cumulants of the free energy, see Corollary \ref{Cor: exact formula for cumulant of free energy}. 

\subsection{Outline of the paper}
In Section \ref{sec: prelim}, after establishing some basic notation, we recall the Mellin transform framework introduced in \cite{CN2018}, where it was noticed that the four basic beta-gamma models can be treated simultaneously.

In Section \ref{section DRP}, we recall the "down-right" property shared by the four basic beta-gamma models. This is a  consequence of the stronger Burke property, and implies in particular that the free energy can be written as the sum of two i.i.d. sums of order $O(N)$ see \eqref{eqn: NSEW}. Understanding the fluctuations of the free energy becomes equivalent to understanding the correlation between these two sums, or equivalently the correlation between one of them and the free energy. This is manifested in the expansion for the cumulants of the free energy appearing in Lemma \ref{lemma: cumulant of log Z to joint cumulants}.

In Section \ref{sec: central-moments}, we develop a formula of "integration by parts" type, which expresses certain correlations appearing in the expansion for the cumulant in terms of derivatives of expectations of moments of the free energy with respect to the parameter in the boundary weights. See Lemma \ref{lemma: higher derivatives}. We use this to obtain formulas relating the the cumulants of the free energy to expectations of productions of quenched cumulants of $t_1$, the first jump in the system. See Corollary \ref{Cor: exact formula for cumulant of free energy}.

In Section \ref{sec: estimates for central}, we prove our main result. The key idea here is that the formulas obtained in Section \ref{sec: central-moments} allow one to get improved estimates (compared to the trivial $O(N^{1/2})$ for the moments of the free energy given estimates for the annealed moments of $t_1$. See Lemma \ref{improved central moment}. Conversely, an inequality due essentially to Sepp\"al\"ainen \cite{S} relates the moments of $t_1$ to moments of the centered free energy. Iterating through these two inequalities a finite number of times, we can obtain bounds that are arbitrarily close to order $N^{1/3}$.

\section{Preliminaries and notation}\label{sec: prelim}
\subsection{Notation}
We let $\N= \{1,2,\ldots\}$, $\Z_+ = \{0,1,\ldots\}$, while $\R$ denotes the real numbers. 

Let $\floor{x}$ denote the greatest integer less than or equal to $x$. Let $\vee$ and $\wedge$ denote maximum and minimum, respectively:
\begin{align*}
    a\vee b &= \max\{a,b\},\\
    a\wedge b &= \min\{a,b\}.
\end{align*}

Given a real valued function $f$, let $\supp(f)= \{x: f(x)\neq 0\}$ denote the support of the function $f$ (note that we do not insist on taking the closure of this set). 

Given a random variable $X$ with finite expectation, we let 
\[\overline{X}= X-\E[X].\]
The symbol $\otimes$ is used to denote (independent) product distribution.

\subsection{The Mellin transform framework}\label{section Mellin}
Here we introduce a framework, developed in \cite{CN2018}, which allows us to treat the four basic beta-gamma models simultaneously.

Given a function $f:(0,\infty)\rightarrow[0,\infty)$, write $M_f$ for its Mellin transform
 \[
M_f(a):= \int_0^\infty x^{a-1}f(x)\mathrm{d}x
\]
for any $a\in\R$ such that the integral converges. Define 
\[
D(M_f):=\text{interior}(\{a\in\R : 0<M_f(a)<\infty\}).
\]
\begin{definition}
Given a function $f:(0,\infty)\to [0,\infty)$ such that $D(M_f)$ is non-empty, we define a family of densities on $(0,\infty)$ parametrized by $a\in D(M_f)$:
\begin{equation}\label{equation Mellin density}
\rho_{f,a}(x):= M_f(a)^{-1}x^{a-1}f(x).
\end{equation}
We write $X\sim m_f(a)$ to denote that the random variable $X$ has density $\rho_{f,a}$.
\end{definition}

\begin{remark}\label{remark - Mellin consequences}
If $f:(0,\infty)\to [0,\infty)$ is such that $D(M_f)$ is non-empty, then $M_f$ is $C^\infty$ throughout $D(M_f)$. Furthermore, if $X\sim m_f(a)$, then 
\begin{enumerate}
\item $\log X$ has finite exponential moments. That is, there exists some $\epsilon>0$ such that 
\[
\E[e^{\epsilon |\log X|}]\leq \E[X^{\epsilon}]+\E[X^{-\epsilon}]=\frac{M_f(a+\epsilon)+M_f(a-\epsilon)}{M_f(a)}<\infty.
\]
\item For all $k\in \N$,
\[\frac{\p^k}{\p a^k} M_f(a)=M_f(a) \E[(\log X)^k].\]
\item $\kappa_k(\log X)=\psi_k^f(a)$, where 
\[
\psi^f_k(a):=\frac{\p^{k+1}}{\p a^{k+1}}\log{M_f(a}) \text{ for } k\in\mathbb{Z}_+.
\]
\end{enumerate}
\end{remark}

%The following remark says that random variables with densities of the form \eqref{equation Mellin density} are closed under inversion.

%\begin{remark}\label{remark - Mellin inversion}
%If $f:(0,\infty)\rightarrow [0,\infty)$ is such that $D(M_f)$ is non-empty and $g(x):=f(\frac{1}{x})$ for $x\in (0,\infty)$, then for all $a\in D(M_f)$,
%\begin{enumerate}
%\item  $X\sim m_f(a) \Leftrightarrow X^{-1}\sim m_g(-a)$,
%\item $M_f(a)=M_g(-a)$  and therefore  $D(M_g)=-D(M_f)$,
%\item  $\psi_k^f(a)=(-1)^{k+1} \psi_k^g(-a)$ for all $k\in \Z_+$.
%\end{enumerate}
%\end{remark}

\subsection{The four basic beta-gamma models are Mellin-type}\label{sec: four mellin}

The random variables appearing in each of the four basic beta-gamma models have densities of the form \eqref{equation Mellin density}, for various choices of $f$, which we specify here. In the table below, we assume $b>0$ and $a\in D(M_f)$.
\begin{figure}[H]
\begin{equation*}
\begin{array}{|c|c|}
		\hline  f(x) &  m_f(a)   \\
		\hline \hline
		 e^{-bx} &  \text{Ga}(a,b) \\ \hline
		 e^{-b/x} & \text{Ga}^{-1}(-a,b) \\ \hline
		 (1-x)^{b-1}\ind_{\{0<x<1\}} &\text{Be}(a,b) \\ \hline
		 (1-\frac{1}{x})^{b-1}\ind_{\{x>1\}} &\text{Be}^{-1}(-a,b) \\ \hline
		 (\frac{x}{x+1})^{b} & \text{Be}^{-1}(-a,b+a)-1 \\ \hline
	\end{array}
\end{equation*}
% \vspace{6pt}
% \begin{equation*}
% \begin{array}{|c|c|}
% \hline f(x) & \psi_n^f(a)\\
% \hline \hline
% e^{-bx} & \Psi_n(a) - \delta_{n,0}\log b  \\ \hline
% e^{-b/x} & (-1)^{n+1}(\Psi_n(-a)-\delta_{n,0}\log b) \\ \hline
% (1-x)^{b-1}\ind_{\{0<x<1\}} & \Psi_n(a)-\Psi_n(a+b)   \\ \hline
% (1-\frac{1}{x})^{b-1}\ind_{\{x>1\}} & (-1)^{n+1}(\Psi_n(-a)-\Psi_n(-a+b))   \\ \hline
% (\frac{x}{x+1})^{b} & \Psi_n(a+b) + (-1)^{n+1}\Psi_n(-a) \\ \hline
% \end{array}
% \end{equation*}
% \caption{Mellin framework data for the distributions appearing in our basic beta-gamma models.}
 \label{table f}
\end{figure}

To express the distribution of the polymer environment in each of the four models given in \eqref{model-IG} through \eqref{model-IB} within the above framework, we let 
\begin{equation}\label{polymer environment distribution}
	(R^1,R^2,X) \sim m_{f^1}(a_1) \otimes m_{f^2}(a_2) \otimes m_{f^1}(a_3),
\end{equation}
where the functions $f^1$, $f^2$ and parameters $a_j$, $j=1,2,3$ are given in the table in Figure \ref{table 4 models}. Recall that in each of the models, $(Y^1,Y^2)$ are given in terms of $X$. For Table \ref{table 4 models} we assume $\mu,\beta>0$.
\begin{figure}[H] 
\[
	\begin{array}{|l||c|c|c|l|}
		\hline
		\text{Model} & f^1(x) & f^2(x) & (a_1,a_2,a_3)& \\ \hline \hline
		\text{IG} & e^{-\beta/x} & e^{-\beta/x} & (\theta-\mu,-\theta,-\mu)& \theta\in(0,\mu) \\ \hline 
		\text{G} & e^{-\beta x} & (1-\frac{1}{x})^{\mu-1}\ind_{\{x>1\}} & (\mu +\theta,-\theta,\mu) & \theta\in(0,\infty)\\ \hline
		\text{B} & (1-x)^{\beta-1}\ind_{\{0<x<1\}} & (1-\frac{1}{x})^{\mu-1}\ind_{\{x>1\}} & (\mu +\theta,-\theta,\mu)& \theta\in (0,\infty) \\ \hline
		\text{IB} &  (1-\frac{1}{x})^{\beta-1}\ind_{\{x>1\}} & (\frac{x}{x+1})^{(\beta+\mu)} & (\theta-\mu,-\theta,-\mu)&\theta\in(0,\mu) \\ \hline 
	\end{array}
\]
\caption{Functions and parameters to fit the four basic beta-gamma models into the Mellin framework.}
\label{table 4 models}
\end{figure}

% \begin{figure}[H] 
% \[
% 	\begin{array}{|l||c|c|c|l|}
% 		\hline
% 		\text{Model} & f^1(x) & f^2(x) & (a_1,a_2,a_3)& \\ \hline \hline
% 		\text{IG} & e^{-\beta/x} & e^{-\beta/x} & (\theta-\mu,-\theta,-\mu)& \theta\in(0,\mu) \\ \hline 
% 		\text{G} & e^{-\beta x} & (1-\frac{1}{x})^{\mu-1}\ind_{\{x>1\}} & (\mu +\theta,-\theta,\mu) & \theta\in(0,\infty)\\ \hline
% 		\text{B} & (1-x)^{\beta-1}\ind_{\{0<x<1\}} & (1-\frac{1}{x})^{\mu-1}\ind_{\{x>1\}} & (\mu +\theta,-\theta,\mu)& \theta\in (0,\infty) \\ \hline
% 		\text{IB} &  (1-\frac{1}{x})^{\beta-1}\ind_{\{x>1\}} & (\frac{x}{x+1})^{(\beta+\mu)} & (\theta-\mu,-\theta,-\mu)&\theta\in(0,\mu) \\ \hline 
% 	\end{array}
% \]
% \caption{Functions and parameters to fit the four basic beta-gamma models into the Mellin framework.}
% \label{table 4 models}
% \end{figure}

When the polymer environment is as in \eqref{polymer environment distribution} with parameters $(a_1, a_2)$, we use $\P^{(a_1,a_2)}$, $\E^{(a_1,a_2)}$, $\Var^{(a_1,a_2)}$, $\Cov^{(a_1,a_2)}$ in place of $\P$, $\E$, $\Var$, $\Cov$ respectively.

\begin{remark}\label{remark parameter DR property}
For each fixed value of the bulk parameter $a_3$, we obtain a family of models with boundary parameters $a_1$ and $a_2$ satisfying $a_1 + a_2 = a_3$.
\end{remark}

\section{The down-right property}\label{section DRP}
Write $\alpha_1=(1,0)$, $\alpha_2=(0,1)$. For $k=1,2$ define ratios of partition functions
\begin{align*}
R^k_x:=\frac{Z_x}{Z_{x-\alpha_k}}\qquad \text{for all  }x\text{ such that }  x-\alpha_k\in\mathbb{Z}_+^2.
\end{align*}
Note that these extend the definitions of $R^1_{i,0}$ and $R^2_{0,j}$, since for example $Z_{i,0}=\prod_{k=1}^{i} R^1_{k,0}$.
We say that $\pi=\{\pi_k\}_{k\in\mathbb{Z}}$ is a down-right path in $\Z_+^2$ if $\pi_k\in\mathbb{Z}_+^2$ and $\pi_{k+1}-\pi_k\in\{\alpha_1, -\alpha_2\}$ for each $k\in \Z$. To each edge along a down-right path we associate the random variable
\[
\Lambda_{\{\pi_{k-1},\pi_k\}} :=
\begin{cases}
R^1_{\pi_k} & \text{ if } \{\pi_{k-1},\pi_k\} \text{ is horizontal,}\\
R^2_{\pi_{k-1}} & \text{ if } \{\pi_{k-1},\pi_k\} \text{ is vertical}.
\end{cases}
\]
The following definition is a weaker form of the Burke property, see \cite[Theorem 3.3]{S}. 
\begin{definition}\label{def: down right property}
Say the polymer model has the {\rm down-right property} if for all down-right paths $\pi=\{\pi_k\}_{k\in\mathbb{Z}}$, the random variables 
\begin{equation*}
\Lambda(\pi) := \{\Lambda_{\{\pi_{k-1},\pi_k\}}: k\in\mathbb{Z}\}
\end{equation*}
are independent and each $R^1_{\pi_k}$ and $R^2_{\pi_k}$ appearing in the collection are respectively distributed as $R^1$ and $R^2$.
\end{definition}
\begin{proposition}\label{proposition 4 models have DR property}
	Each of the four basic beta-gamma models, \eqref{model-IG} through \eqref{model-IB}, possesses the down-right property.
\end{proposition}
See Proposition 2.3 of \cite{CN2018}.

\subsection{Consequences of the down-right property}\label{subsec: DRP consequences}
The free energy has two useful expressions.
\begin{equation}
\log Z_{m,n}=W+N=S+E\label{eqn: NSEW}
\end{equation}
where
\[ 
W_n:= \sum_{j=1}^n \log R^2_{0,j},\quad E_m:=\sum_{j=1}^n \log R^2_{m,j},\quad  N_n:=\sum_{i=1}^m \log R^1_{i,n}, \quad S_m:=\sum_{i=1}^m \log R^1_{i,0}.    
\]
Notice that $W_n=\log Z_{0,n}$ and $S_m=\log Z_{m,0}$.  If the model possesses the down-right property (see Definition \ref{def: down right property}), then  $W_n, E_m, N_n, S_m$ are each sums of i.i.d.\  random variables.   
%Moreover, they satisfy  $W\overset{d}{=}E$, $N\overset{d}{=}S$, $W\perp S$, and $N\perp E$.

Recall that if the random variables $X_1,\dots, X_k$ have finite exponential moments, then their joint cumulant is defined by  
\begin{equation}
\kappa(X_1,\dots,X_k):=\frac{\partial^k}{\partial\xi_1\dots\partial\xi_k}\log \E[e^{\sum_{j=1}^k\xi_j X_j}]\Big|_{\xi_i=0}.\label{eq: joint culumant formula 1}
\end{equation}

Alternatively, the joint cumulant can be written as a combination of products of expectations of the underlying random variables:
\begin{equation}
\kappa(X_1,\dots,X_k)=\sum_{\pi\in\mathcal{P}}(|\pi|-1)!(-1)^{|\pi|-1}\prod_{B\in \pi}\E\left[\prod_{i\in B} X_i \right]\label{eq: joint cumulant formula 2}
\end{equation}
where $\mathcal{P}$ ranges over partitions $\pi$ of $\{1,\ldots, k\}$ and $|A|$ stands for the size of the set $A$.
In the case where $X_1=X_2=\dots =X_k=X$, the joint cumulant reduces to the $k$-th cumulant of $X$ which we denote by $\kappa_k(X)$.   
\begin{lemma}\label{lemma: cumulant of log Z to joint cumulants} 
Assume the polymer environment is such that $|\log R^1|$, $|\log R^2|$, $|\log Y^1|$, and $|\log Y^2|$ all have finite exponential moments.  
Then, for any positive integer $k$, 
\begin{align}
    \kappa_k(\log Z_{m,n})&=\kappa_k(E_n)-(-1)^k\kappa_k(S_m)-\sum_{j=1}^{k-1} \binom{k}{j}(-1)^{k-j}\kappa(\underbrace{\log Z_{m,n},\cdots,\log Z_{m,n}}_{j \text{ times}},\underbrace{S_m,\cdots,S_m}_{k-j \text{ times}})\,\,\text{ and}\label{eq: cumulant lemma 1}\\
    \kappa_k(\log Z_{m,n})&=\kappa_k(N_m)-(-1)^k\kappa_k(W)-\sum_{j=1}^{k-1} \binom{k}{j}(-1)^{k-j}\kappa(\underbrace{\log Z_{m,n},\cdots,\log Z_{m,n}}_{j \text{ times}},\underbrace{W_n,\cdots,W_n}_{k-j \text{ times}}).\label{eq: cumulant lemma 2}
\end{align}
Moreover, if the polymer model also possesses the down-right property, then
\begin{align*}
    \kappa_k(E_n)&=\kappa_k(W_n)=n\kappa_k(R^1)\quad \text{ and }\\
    \kappa_k(N_m)&=\kappa_k(S_m)=m\kappa_k(R^2).
\end{align*}
\end{lemma}
% \textcolor{red}{(This remark assumes the notation $\psi_k^f$ has already been introduced, which it has not.)}
% Remark: $\kappa_k(N)=\kappa_k(S)=m\psi_{k-1}^{f^1}(\theta)$, and $\kappa_k(E)=\kappa_k(W)=n\psi_{k-1}^{f^2}(\lambda-\theta)$.  so when $k=2$, we recover 
% \begin{align*}
% \mathbb{V}\textrm{ar}(\log Z_{m,n})&=\underbrace{n\psi_1^{f^2}(\lambda-\theta)-m\psi_1^{f^1}(\theta)}_{\text{characteristic direction}}+2\mathbb{C}\textrm{ov}(\log Z_{m,n},S)\,\, \text{ and }\\
% \mathbb{V}\textrm{ar}(\log Z_{m,n})&=\overbrace{m\psi_1^{f^1}(\theta)-n\psi_1^{f^2}(\lambda-\theta)}+2\mathbb{C}\textrm{ov}(\log Z_{m,n},W).
% \end{align*}

% \textcolor{red}{Should we keep the following?  It shows how we can recover the exact statement made in CN path-fluct paper when $k=2$.)}
% Note that, by the \textcolor{red}{Burke property}, $\mathbb{C}\textrm{ov}(\log Z_{m,n},E)=\Cov(N,S)$ and $\mathbb{C}\textrm{ov}(\log Z_{m,n},W)=\Cov(E,W)$, so we recover the equations in Seppalainen.

\begin{proof}
By Lemma \ref{lemma: fin exp. moments for log Z}, $\log Z_{m,n},N_m,S_m,E_n,W_n$ all have finite exponential moments, so their cumulants and joint cumulants exist.    By \eqref{eqn: NSEW}, $E_n=\log Z_{m,n}-S_n$.
Since the joint cumulant is multi-linear, 
\[
\kappa_k(E_n)= \sum_{j=0}^k\binom{k}{j}(-1)^{k-j}\kappa(\underbrace{\log Z_{m,n},\cdots,\log Z_{m,n}}_{j \text{ times}}, \underbrace{S_m,\cdots,S_m}_{k-j \text{ times}}).
\]
The $j=k$ term in the summand is $\kappa_k(\log Z_{m,n})$ and the $j=0$ term is $(-1)^k\kappa_k(S_m)$.  Rearranging yields equation \eqref{eq: cumulant lemma 1}.  To obtain equation \eqref{eq: cumulant lemma 2}, apply the same argument with $(N_m,W_n)$ in place of $(E_n,S_m)$.  For the last part of the Lemma use the fact that $\kappa_k(X+Y)=\kappa_k(X)+\kappa_k(Y)$ if $X$ and $Y$ are independent.
\end{proof}
\begin{remark}
Each of the four basic beta-gamma models satisfy the moment conditions of Lemma \ref{lemma: cumulant of log Z to joint cumulants}.
\end{remark}

\section{Formulas for the central moments}\label{sec: central-moments}
In the next two sections, we give an exact formula for the terms appearing in the summands on the right-hand side of \eqref{eq: cumulant lemma 1}.  The same arguments can be used  to obtain analogous estimates for \eqref{eq: cumulant lemma 2}, but as will be apparent in Section \ref{sec: estimates for central} the estimates in \eqref{eq: cumulant lemma 1}  will be sufficient to obtain Theorem \ref{theorem: Main 1}.
\subsection{Integration by parts type formula}

Referring to the notation in Section \ref{section Mellin}, fix an integer $r\geq 1$ and let $f_k:(0,\infty)\rightarrow [0,\infty)$ for $k=1,\dots,r$ and $a_0<a<a_1$ be real numbers such that  $[a_0,a_1]\subset\cap_{k=1}^r D(M_{f_k})$. Consider a collection of independent random variables $\{X_k\}_{k=1}^r$ where $X_k\sim m_{f_k}(a)$ for all $ 1\leq k\leq r$, and let $\E^{a}$ corresponds to the expectation over these random variables.
Finally, define
\[T:=\sum_{k=1}^r \log{X_k}.\]

We introduce a sequence $\{p_n(t,a;r)\}_{n\ge 0}$ of $n$-th degree polynomials in $s$ defined recursively by
\begin{align}
p_0(t,a;r)&= 1 \nonumber \\
p_n(t,a;r)&=\frac{\p}{\p a}p_{n-1}(t,a;r)+p_{n-1}(t,a;r)(t-\E^a[T]) \text{ for } n\geq 1.\label{eqn: recursion 2}
\end{align}
Note that the dependence of $p_n(\cdot,\cdot;r)$ on $r$ is polynomial, a fact we use explicitly later in our argument (see Proposition \ref{prop: poly-bound}).

The following lemma extends Lemma B.2 from \cite{CN2018}.
%and gives us a way to analyze the expectations that appear in the \textcolor{red}{right hand side of Lemma} \ref{lemma: joint cumulants to polynomials}.
\begin{lemma}\label{lemma: higher derivatives}
 $A:\R^r\rightarrow\R$ be a measurable function such that $\E^{a}[A(X_1,\cdots,X_r)^2]<\infty$ for all $a\in [a_0,a_1]$.  Then 
\begin{align}\label{eqn: ibp-formula-A}
\frac{\p^n}{\p a^n}\E^{a}[A(X_1,\cdots,X_r)] &=\E^{a}[A(X_1,\cdots,X_r)p_n(T,a;r)].
\end{align}

\end{lemma}

\begin{proof}
 The joint density of $(\log X_1 ,\log X_2 ,\dots,\log X_r)$ is given by 
\[
g(x_1,\dots,x_r)=\frac{e^{a\sum_{k=1}^rx_k}}{\prod_{k=1}^rM_{f_k}(a)}\prod_{k=1}^rf_k(e^{x_k}).
\]
Thus the density of $T=\sum_{k=1}^r\log X_k$ is 
\begin{align}
h_a(t)=\frac{e^{at}}{\prod_{k=1}^rM_{f_k}(a)}\int_{\mathbb{R}^{r-1}}f_1(e^{x_1})f_2(e^{x_2-x_1})\dots f_r(e^{t-x_{r-1}})\mathrm{d}x_1,\dots,x_{r-1}.\label{eq:h}
\end{align}
Therefore, the joint density of $(\log X_1 ,\log X_2 ,\dots,\log X_r )$ given that $T=t$ is given by 
\begin{equation}
\frac{g(x_1,\dots,x_r)\mathbbm{1}_{\{\sum_{k=1}^rx_k=t\}}}{h_a(t)}=\frac{\prod_{k=1}^rf_k(e^{x_k})}{\int_{\mathbb{R}^{r-1}}f_1(e^{x_1})f_2(e^{x_2-x_1})\dots f_r(e^{t-x_{r-1}})\mathrm{d}x_1,\dots,x_{r-1}}\label{eqn: conditional density},
\end{equation}
which has no $a$-dependence.  Recursion \eqref{eqn: recursion 2} and $\frac{\p}{\p a}h_a(t)=h_a(t)(t-\E^a[T])$ inductively imply
\begin{align}
\frac{\p^n}{\p a^n}h_a(t)=h_a(t)p_n(t,a;r) \text{ for all } n\in \mathbb{Z}_+.\label{h recursion}
\end{align}
  By \eqref{eqn: conditional density} and \eqref{h recursion},
\begin{align*}
\frac{\p^n}{\p a^n}\mathbb{E}^a[A(X_1,\cdots,X_k)] &=\frac{\p^n}{\p a^n}\int_{\mathbb{R}}\mathbb{E}^a[A(X_1,\cdots,X_k)|T=t]h_a(t)\mathrm{d}t\\
\ &=\int_{\mathbb{R}}\mathbb{E}^a[A(X_1,\cdots,X_k)|T=t]\frac{\p^n}{\p a^n}h_a(t)\mathrm{d}t\\
\ &=\int_{\mathbb{R}}\mathbb{E}^a[A(X_1,\cdots,X_k)|T=t]h_a(t)p_n(t,a;r)\mathrm{d}t\\
\ &=\E^a\Bigl(A(X_1,\cdots,X_k)p_n(T,a;r)\Bigr).
\end{align*}
The interchanging of the  $n$-th derivative and the integral will be justified by the bound:
\begin{align}
\int_{\mathbb{R}}\mathbb{E}[\left|A(\{X_k\}_{k=1}^r)\right|\big|T=t]\sup_{a\in[a_0,a_1]}\left|\frac{\p^n}{\p a^n}h_a(t)\right|\mathrm{d}t<\infty. \label{eq:sup deriv}
\end{align}
To obtain this bound first notice that recursion \eqref{eqn: recursion 2}
%\begin{align}
%p_0(s,a)&= 1\\
%p_n(s,a)&=\frac{\p}{\p a}p_{n-1}(s,a)+p_{n-1}(s,a)(s-\E^a[S]) \text{ for } n\geq 1.
%\end{align}
implies that $p_n(t,a)$ are degree $n$ polynomials in $t$ with coefficients that are smooth in a. Thus, by \eqref{h recursion}, there exist constants $0<C_n<\infty$ independent of $t$ such that 
\begin{align}
\sup_{a\in[a_0,a_1]}\left|\frac{\p^n}{\p a^n }h_a(t)\right|\leq C_n(1+|t|)^{n}\sup_{a\in[a_0,a_1]}h_a(t). \label{sup bound 1}
\end{align}
Once we show that there is a constant $C$ depending only on $a_0$ and  $a_1$ such that 
\begin{align}
\sup_{a\in[a_0,a_1]}h_a(t)\leq h_{a_0}(t)+Ch_{a_1}(t)\quad \text{ for all } s\in \R, \label{eq:sup bound}
\end{align}
\eqref{sup bound 1} will give the bound \eqref{eq:sup deriv} since
\begin{align*}
\int_{\mathbb{R}}\mathbb{E}[|A(\{X_k\}_{k=1}^r)|\mid T=t](1+|t|)^{n} h_{a_j}(t)\mathrm{d}t=\mathbb{E}^{a_j}[|A(\{X_k\}_{k=1}^r)|(1+|T|)^{n}]\\
\leq \mathbb{E}^{a_j}[(A(\{X_k\}_{k=1}^r))^2]^\frac{1}{2}\mathbb{E}^{a_j}[(1+|T|)^{2n}]^\frac{1}{2}.
\end{align*}
This is finite since $\E^{a_j}[A(\{X_k\}_{k=1}^r)^2]<\infty$ by assumption, and $T$ is a sum of independent random variables with finite exponential moments.
All that is left to do is verify the bound \eqref{eq:sup bound}.  To accomplish this, notice that equation \eqref{eq:h} implies 
\begin{equation}
\label{eqn: log h}
\frac{\p}{\p a}\log{h_a(t)}=t-\E^a[T].
\end{equation}
Since  $\E^a[T]=\sum_{k=1}^r\psi_0^{f_k}(a)$, $a\mapsto \E^a[T]$ is an increasing function (recall that $\frac{d}{da}\psi_0^{f_k}(a)=\psi_1^{f_k}(a)=\Var[X_k]>0$).  Therefore, for all $t\leq\E^{a_0}[T]$, the function $a\mapsto h_a(t)$ is non-increasing on $[a_0,a_1]$ which gives 
\[
\sup_{a\in[a_0,a_1]}h_a(t)\leq h_{a_0}(t)\quad  \text{ for all }  t\leq \E^{a_0}[T].
\]
On the other hand, if $t> \E^{a_0}[T]$, then 
\begin{equation}\label{eqn: suph}
\frac{\p}{\p a}\log \Bigl(h_a(t)\exp{\bigl(a(\E^{a_1}[T]-\E^{a_0}[T])}\bigr)\Bigr)=t-\E^a[T]+\E^{a_1}[T]-\E^{a_0}[T]> 0
\end{equation}
for all $a\in[a_0,a_1]$.  Thus, for all $t> \E^{a_0}[T]$, 
\[a\mapsto h_a(t)\exp{\Bigl(a(\E^{a_1}[T]-\E^{a_0}[T])\Bigr)}\] is increasing on the interval $[a_0,a_1]$.  Therefore
\[
\sup_{a\in[a_0,a_1]}h_a(t)\leq C_3h_{a_1}(t) \text{ for all } t> \E^{a_0}[T]
\]
where $C=\exp{\Bigl((a_1-a_0)(\E^{a_1}[T]-\E^{a_0}[T]})\Bigr)$.  Combining \eqref{eqn: log h} and \eqref{eqn: suph}  gives the desired result.
\end{proof}

\begin{lemma}\label{lemma: joint cumulants to polynomials} Assume the polymer environment satisfies $R^1_{i,0}\sim m_f(a)$ for all $i\geq 1$. Let $k\geq 2$ and $1\leq j\leq k$.  For $r\geq 1$, let
\[
S_r:=\sum_{i=1}^r R^1_{i,0}.
\]

Then, 

\begin{equation}
\label{eqn: renormalized-cumulant}
\begin{split}
\kappa(&\underbrace{\log Z_{m,n},\dots,\log Z_{m,n}}_{j\text{ times}},\underbrace{S_r,\dots,S_r}_{k-j \text{ times}})=\kappa(\underbrace{\ov{\log Z_{m,n}},\dots,\ov{\log Z_{m,n}}}_{j\text{ times}},\underbrace{S_r,\dots,S_r}_{k-j \text{ times}})\\
 &=
\sum_{\pi\in \mathcal{P}} (|\pi|-1)!(-1)^{|\pi| -1} \prod_{B\in \pi} \mathbb{E}\left[(\ov{\log Z_{m,n}})^{|B \cap\{1,\ldots, j\}|}p_{|B \cap \{j+1,\ldots,k\}|}(S_r,a;r)\right] +r\psi^f_{k-j}(a)
\end{split}
\end{equation}
where $\mathcal{P}$ ranges over partitions $\pi$ of $\{1,\ldots, k\}$ such that no block $B\in \pi$ is contained in $\{j+1,\ldots, k\}$, and $p_k(s,a)$ are polynomials recursively defined by \eqref{eqn: recursion 2} in the case $f_j=f$ for all $j$.
\end{lemma}
\begin{proof}
%\textcolor{red}{by the recurrence
%\begin{equation}p_k(s,a)=\frac{\partial}{\partial a}p_{k-1}(s,a)+p_{k-1}(s,a)(s-m\psi_0^f(a)),\label{eqn: poly recursion}\end{equation}}
%where $p_0(s,a)=1$.
Introduce the function 
\[g(s,a;r):=e^{as-r\log M_f(a)}.\]
Note that
\[\frac{\partial}{\partial a} e^{as-r\log M_f(a)}=(s-r \psi_0^f(a))g(s,a;r),\]
so
\begin{align*}
\frac{\partial}{\partial a}\big( g(s,a;r) p_{k-1}(s,a;r)\big)&= g(s,a;r)\frac{\partial}{\partial a}p_{k-1}(s,a;r)+g(s,a;r)p_{k-1}(s,a;r)(s-r\psi_0^f(a))\\
&=g(s,a;r)p_k(s,a;r).
\end{align*}
Rearranging, this gives
\begin{align*}
    p_k(s,a;r)&=\frac{1}{g(s,a;r)}\frac{\partial}{\partial a}\big(g(s,a;r)p_{k-1}(s,a;r)\big)\\
    &=...\\
    &= \frac{1}{g(s,a;r)}\frac{\partial^l}{\partial a^l}(g(s,a;r) p_{k-l}(s,a;r)).
\end{align*}
Letting $l=k$, we have
\begin{equation}
    p_k(s,a;r)= \frac{\frac{\partial^k}{\partial a^k }g(s,a;r)}{g(s,a;r)}.\nonumber
\end{equation}
By Taylor equation expansion, we have
\[g(s,a+\lambda;r)=\sum_{k=0}^\infty \frac{\lambda^k}{k!}\frac{\partial^k}{\partial a^k}g(s,a;r),\]
which yields the generating function for the polynomials $p_k(s,a;r)$:
\begin{equation}
\frac{g(s,a+\lambda;r)}{g(s,a;r)}=e^{\lambda s} \left(\frac{M_f(a)}{M_f(a+\lambda)}\right)^r =\sum_{k=0}^\infty \frac{\lambda^k}{k!}p_k(s,a;r).\label{eqn: generating function}
\end{equation}

Using this, we have a formula for joint cumulants, as follows. Recall
\[\kappa_k(\underbrace{A,\ldots, A}_{j \text{ times}},\underbrace{S_r,\ldots, S_r}_{k-j \text{ times}})=\frac{\p}{\p\xi_1\cdots \p\xi_k}\log \mathbb{E}[e^{(\xi_1+\cdots+\xi_j)A }e^{(\xi_{j+1}+\cdots+ \xi_k)S_r}]\Big|_{\xi_i=0}.\]
Inserting the generating function \eqref{eqn: generating function}, we have
\[e^{(\xi_{j+1}+\cdots+\xi_k) S_r}= \big(\frac{M_f(a+\xi_{j+1}+\cdots \xi_k)}{M_f(a)}\big)^r\sum_{l=0}^\infty \frac{(\xi_{j+1}+\cdots+\xi_k)^l p_l(S_r,a;r)}{l!}\]
Taking expectations, then logarithms, we have 
\[\log \mathbb{E}[e^{(\xi_1+\cdots+ \xi_j)A} \sum_{l=0}^\infty \frac{\left(\xi_{j+1}+\cdots+\xi_k\right)^l}{l!}p_l(S_r,a;r)]-r\log M_f(a)+r\log M_f(a+\xi_{j+1}+\cdots+\xi_k)\].

Setting $A=\ov{\log Z_{m,n}}$ and taking derivates with respect to the $\xi_i$'s, then evaluating them at zero gives \eqref{eqn: renormalized-cumulant}.

The only part of the statement that still requires comment is the assertion that partitions $\mathcal{P}$ with a block contained in $\{j+1,\ldots,k\}$ make a zero contribution. This is because
\[\mathbb{E}[p_n(S_r,a;r)]=0\]
for $n\ge 1$, as follows from \eqref{eqn: ibp-formula-A} with $A\equiv 1$.
\end{proof}

\subsection{Coupling of polymer environments}\label{sec: coupling}
In order to compare polymer environments with different parameters, we use a coupling to express the boundary weights as functions of i.i.d.\ uniform$(0,1)$ random variables.

Recall the notation from Section \ref{section Mellin}.  Suppose $f:(0,\infty)\rightarrow [0,\infty)$ is a smooth function on its support, $\supp(f)$ is open, and $D(M_f)$ is non-empty.  Define $F^f: D(M_f)\times (0,\infty)\rightarrow [0,1]$ by
\[
F^f(a,x):=\frac{1}{M_f(a)}\int_0^x y^{a-1}f(y)\mathrm{d}y.
\]
For fixed $a\in D(M_f)$, $x\mapsto F^f(a,x)$ is the
cdf of a random variable with $X\sim m_f(a)$. Note that $F^f(a,\cdot)$ is a bijection between $\supp(f)$ and $(0,1)$. For $a\in D(M_f)$, let $H^f(a,\cdot)$ be the inverse of $F^f(a,\cdot)$ defined on $(0,1)$.  By the inverse function theorem, $H^f:D(M_f)\times (0,1)\rightarrow \supp(f)$ is a smooth function in both of its variables satisfying
\begin{equation}
\frac{\p}{\p a}\log H^f(a,x)=L^f(a,H^f(a,x)),\label{eqn: H-deriv}
\end{equation}
where 
\[
L^f(a,x):=\frac{1}{x^af(x)}\int_0^xy^{a-1}(\psi^f_0(a)-\log y) f(y)\mathrm{d}y.
\]

Another expression for $L^f$ shows that it is a strictly positive function:
\begin{equation}
L^f(a,x)=\frac{-1}{x^a f(x)}\Cov(\log X, \ind_{\{X\leq x\}})>0.\label{eqn: Lf is positive}
\end{equation}
Since $f$ is smooth, $L^f$ is smooth as a function on $D(M_f)\times\supp(f)$.
Note that if $\eta$ is a uniform$(0,1)$ distributed random variable, then $H^f(a,\eta)\sim m_f(a)$ \underline{for every} $a\in D(M_f)$.  This gives us a useful coupling as follows.

 Fix $m,n\in \N$ and an environment $\omega$ on the square with lower-left corner $(0,0)$ and upper-right corner $(m,n)$.  Assume the random variables attached to the southern boundary $R^1_{i,0}$ all have $m_f(a)$ distributions. Fix $1\leq r\leq m$ and  let $\{\eta_i\}_{i=1}^r$ be i.i.d.\ uniform$(0,1)$ distributed random variables which are also independent of the original environment $\omega$.   Now create a new environment $\tilde{\omega}$ by replacing $R^1_{i,0}$ in the original environment along the southern boundary by $\tilde{R}^1_{i,0}:=H^f(a,\eta_i)$ only for $i=1,\dots, r$.  This new environment is equal in distribution to the old one, $\tilde{\omega}\overset{d}{=}\omega$.  Write 
\[
Z_{m,n}(a):= Z_{m,n}^{\tilde{\omega}}=\sum_{x_{\cdot}\in \Pi_{m,n}}\prod_{i=1}^{m+n}\tilde{\omega}_{(x_{i-1},x_i)} =\sum_{x_{\cdot}\in \Pi_{m,n}}\prod_{i=1}^{t_1(x_{\cdot})\wedge r}\tilde{R}^1_{i,0}\prod_{i=t_1(x_{\cdot})\wedge r+1}^{m+n}\omega_{(x_{i-1},x_i)}, 
\]
where $t_1(x_{\cdot}):=\max\{i\geq 0: x_i=(i,0)\}$, i.e.\ the exit time from the southern boundary.
By equation \eqref{eqn: H-deriv}, 
\begin{equation}
\frac{\p}{\p a} \tilde{R}^1_{i,0}=\tilde{R}^1_{i,0}L^f(a,\tilde{R}^1_{i,0}).\label{eqn: R deriv}
\end{equation}
Therefore
\begin{equation}
\frac{\p}{\p a}\log Z_{m,n}(a)=\tilde{E}^a_{m,n}[\sum_{i=1}^{t_1(x_{\cdot})\wedge r}L^f(a,\tilde{R}^1_{i,0})]\label{eqn: Z deriv},
\end{equation}
where $\tilde{E}^a_{m,n}[f(x_{\cdot})]$ is defined to be the expectation of the function $f$ of up-right paths $x_{\cdot}$ under the quenched probability measure 
\[
\tilde{Q}_{m,n}^a(x_{\cdot}):=\frac{1}{Z_{m,n}(a)}\prod_{i=1}^{m+n}\tilde{\omega}_{(x_{i-1},x_i)}=\frac{1}{Z_{m,n}(a)}\prod_{i=1}^{t_1(x_{\cdot})\wedge r}\tilde{R}^1_{i,0}\prod_{i=t_1(x_{\cdot})\wedge r+1}^{m+n}\omega_{(x_{i-1},x_i)}.
\]

By \eqref{eqn: R deriv} and \eqref{eqn: Z deriv},   
\begin{equation}
\frac{\p}{\p a} \log \tilde{Q}_{m,n}^a(x_{\cdot})=\sum_{i=1}^{t_1(x_{\cdot})\wedge r}L^f(a,\tilde{R}^1_{i,0})-\tilde{E}^a_{m,n}[\sum_{i=1}^{t_1(x_{\cdot})\wedge r}L^f(a,\tilde{R}^1_{i,0})].\label{eqn: Q deriv} 
\end{equation}

\subsection{Derivatives with respect to boundary parameters}
We now use equation \eqref{eqn: Q deriv} to provide a recursion which will help determine $\frac{\p^k}{\p a^k}\log  Z_{m,n}(a)$.

\begin{lemma}
Suppose $G:D(M_f)\times \supp(f)\rightarrow \R$ is a $C^k$ function.  Then the function $\t{\p}G:D(M_f)\times \supp(f)\rightarrow \R$ defined by
\begin{equation}\label{eqn: tilde-der}
\tilde{\p}G(a,x):=\frac{\p G}{\p a}(a,x)+x L^f(a,x) \frac{\p G}{\p x}(a,x) 
\end{equation}
is $C^{k-1}$ and satisfies the two equations:
\begin{align}
\tilde{\p}G(a,\tilde{R}^1_{i,0})&=\frac{\p}{\p a} G(a,\tilde{R}^1_{i,0})\nonumber\\
\frac{\p}{\p a} \tilde{E}^a_{m,n}\left[\sum_{i=1}^{t_1(x_{\cdot})\wedge r}G(a,\tilde{R}^1_{i,0})\right]&=  \tilde{E}^a_{m,n}\left[\sum_{i=1}^{t_1(x_{\cdot})\wedge r}\tilde{\p}G(a,\tilde{R}^1_{i,0})\right]\label{eqn: exact deriv 1}\\
&+\widetilde{Cov}^a_{m,n}\left(\sum_{i=1}^{t_1(x_{\cdot})\wedge r}G(a,\tilde{R}^1_{i,0}),\sum_{i=1}^{t_1(x_{\cdot})\wedge r}L^f(a,\tilde{R}^1_{i,0})\right),\nonumber
\end{align}
where $\widetilde{Cov}^a_{m,n}$ stands for the covariance under $\widetilde{E}^a_{m,n}$.
\end{lemma}

\begin{proof}
By equation \eqref{eqn: R deriv},
\[
\frac{\p}{\p a}G(a,\tilde{R}^1_{i,0})=\frac{\p G}{\p a}(a,\tilde{R}^1_{i,0})+\frac{\p}{\p a}\tilde{R}^1_{i,0}\frac{\p G}{\p x}(a,\tilde{R}^1_{i,0})=G_1 (a,\tilde{R}^1_{i,0}).
\]
\begin{align*}
\tilde{E}^a_{m,n}\left[\sum_{i=1}^{t_1(x_{\cdot})\wedge r}G(a,\tilde{R}^1_{i,0})\right]=\sum_{x_{\cdot}\in \Pi_{m,n}}\left(\sum_{i=1}^{t_1(x_{\cdot})\wedge r}G(a,\tilde{R}^1_{i,0})\right)\tilde{Q}_{m,n}^a(x_{\cdot}).
\end{align*}
Therefore,
\begin{align}
\frac{\p}{\p a}\tilde{E}^a_{m,n}\left[\sum_{i=1}^{t_1(x_{\cdot})\wedge r}G(a,\tilde{R}^1_{i,0})\right]&=\sum_{x_{\cdot}\in \Pi_{m,n}}\frac{\p}{\p a}\left (\left(\sum_{i=1}^{t_1(x_{\cdot})\wedge r}G(a,\tilde{R}^1_{i,0})\right)\tilde{Q}_{m,n}^a(x_{\cdot})\right)
\nonumber\\
&=\sum_{x_{\cdot}\in \Pi_{m,n}} \left(\sum_{i=1}^{t_1(x_{\cdot})\wedge r}\frac{\p}{\p a}G(a,\tilde{R}^1_{i,0})\right)\tilde{Q}_{m,n}^a(x_{\cdot})\label{eqn part: recurrence rel 1}\\
&+\sum_{x_{\cdot}\in \Pi_{m,n}} \left(\sum_{i=1}^{t_1(x_{\cdot})\wedge r}G(a,\tilde{R}^1_{i,0})\right)\frac{\p}{\p a}\tilde{Q}_{m,n}^a(x_{\cdot})\label{eqn part: recurrence rel 2}.
\end{align}

By the first part of the lemma, the right-hand side of \eqref{eqn part: recurrence rel 1} equals $\tilde{E}^a_{m,n}\left[\sum_{i=1}^{t_1(x_{\cdot})\wedge r}G_1(a,\tilde{R}^1_{i,0})\right]$.
Using equation \eqref{eqn: Q deriv}, \eqref{eqn part: recurrence rel 2} equals
\begin{align*}
&\sum_{x_{\cdot}\in \Pi_{m,n}} \left(\sum_{i=1}^{t_1(x_{\cdot})\wedge r}G(a,\tilde{R}^1_{i,0})\right)\left(\sum_{i=1}^{t_1(x_{\cdot})\wedge r}L^f(a,\tilde{R}^1_{i,0})-\tilde{E}^a_{m,n}\left[\sum_{i=1}^{t_1(x_{\cdot})\wedge r}L^f(a,\tilde{R}^1_{i,0})\right]\right)\tilde{Q}_{m,n}^a(x_{\cdot})\\
&=\tilde{E}^a_{m,n}\left[\left(\sum_{i=1}^{t_1(x_{\cdot})\wedge r}G(a,\tilde{R}^1_{i,0})\right)\left(\sum_{i=1}^{t_1(x_{\cdot})\wedge r}L^f(a,\tilde{R}^1_{i,0})\right)\right]\\
&-\tilde{E}^a_{m,n}\left[\sum_{i=1}^{t_1(x_{\cdot})\wedge r}G(a,\tilde{R}^1_{i,0})\right]\tilde{E}^a_{m,n}\left[\sum_{i=1}^{t_1(x_{\cdot})\wedge r}L^f(a,\tilde{R}^1_{i,0})\right]\\
&=\widetilde{Cov}^a_{m,n}\left(\sum_{i=1}^{t_1(x_{\cdot})\wedge r}G(a,\tilde{R}^1_{i,0}),\sum_{i=1}^{t_1(x_{\cdot})\wedge r}L^f(a,\tilde{R}^1_{i,0})\right).
\end{align*}

% \textcolor{blue}{Next, use equation \eqref{eqn: Z deriv} and apply Lemma 22.1 $k-1$-times to the smooth function $L^f(a,x)$ to get and exact formula expressing $\frac{\p^k}{\p a^k}\log Z_{m,n}(a)$ as an element of the "algebra" generated by at most $k-1$ products of the form
% \[
% \sum_{i=1}^{t_1(x_{\cdot})\wedge r} L^f_j(a,\tilde{R}^1_{i,0}) \quad \text{ for } 1\leq j\leq k-1.
% \]
% }
\end{proof}

Letting $\tilde{\kappa}^a_k(X_1,\dots,X_k)$ denote the quenched joint cumulant of $k$ random variables with respect to $\tilde{E}^a_{m,n}$, repeated application of \eqref{eqn: exact deriv 1} and the chain rule give:
\begin{lemma}\label{lemma: deriv of kappa}
For $g_1,\dots, g_k\in C^\infty(D(M_f)\times \supp(f))$,
\begin{align*}
\frac{\p}{\p a} \tilde{\kappa}_k^a&\left(\sum_{i=1}^{t_1\wedge r}g_1(a,\tilde{R}_{i,0}^1),\sum_{i=1}^{t_1\wedge r}g_2(a,\tilde{R}_{i,0}^1),\cdots,\sum_{i=1}^{t_1\wedge r}g_k(a,\tilde{R}_{i,0}^1) \right)\\
=
\sum_{\substack{\delta_1+\dots+ \delta_k=1\\\delta_i\geq 0}}&\tilde{\kappa}_k^a\left(\sum_{i=1}^{t_1\wedge r}\tilde{\p}^{\delta_1}g_1(a,\tilde{R}_{i,0}^1),\sum_{i=1}^{t_1\wedge r}\tilde{\p}^{\delta_2}g_2(a,\tilde{R}_{i,0}^1),\cdots,\sum_{i=1}^{t_1\wedge r}\tilde{\p}^{\delta_k}g_k(a,\tilde{R}_{i,0}^1) \right)\\
+&\tilde{\kappa}_{k+1}^a\left(\sum_{i=1}^{t_1\wedge r}g_1(a,\tilde{R}_{i,0}^1),\sum_{i=1}^{t_1\wedge r}g_2(a,\tilde{R}_{i,0}^1),\cdots,\sum_{i=1}^{t_1\wedge r}g_k(a,\tilde{R}_{i,0}^1),\sum_{i=1}^{t_1\wedge r}L^f(a,\tilde{R}_{i,0}^1) \right).
\end{align*}
\end{lemma}

\begin{proof}
Write $X_l(a):=\sum_{i=1}^{t_1\wedge r}g_l(a,\tilde{R}_{i,0}^1)$ for $l=1,\dots,k$.  We will perturb the parameter in the environment and the parameters in the arguments of the cumulant separately.  To this end, define 
\[
F(a,b):=\tilde{\kappa}_k^b\left(\sum_{i=1}^{t_1\wedge r}g_1(a,\tilde{R}_{i,0}^1),\sum_{i=1}^{t_1\wedge r}g_2(a,\tilde{R}_{i,0}^1),\cdots,\sum_{i=1}^{t_1\wedge r}g_k(a,\tilde{R}_{i,0}^1) \right).
\]
Using equation \eqref{eqn: R deriv} and multi-linearity of the joint cumulant gives 
\[
\frac{\p}{\p a}F(a,b)=\sum_{\substack{\delta_1+\dots+ \delta_k=1\\\delta_i\geq 0}}\tilde{\kappa}_k^b\left(\sum_{i=1}^{t_1\wedge r}\tilde{\p}^{\delta_1}g_1(a,\tilde{R}_{i,0}^1),\sum_{i=1}^{t_1\wedge r}\tilde{\p}^{\delta_2}g_2(a,\tilde{R}_{i,0}^1),\cdots,\sum_{i=1}^{t_1\wedge r}\tilde{\p}^{\delta_k}g_k(a,\tilde{R}_{i,0}^1) \right).\label{eqn: quenched joint cumulant 1}
\]
Now define $Y(b):=\sum_{i=1}^{t_1\wedge r}L^f(b,K^f(b,\eta_i))$.  Then 
\[
\frac{\p}{\p b} \tilde{E}^b_{m,n}[e^{\sum_{l=1}^k \xi_l X_l(a)}]=\tilde{E}^b_{m,n}[e^{\sum_{l=1}^k \xi_l X_l(a)}\overline{Y(b)}]
\]
where the centering is respect to $\tilde{E}^a_{m,n}$.  Thus
\begin{align}
\frac{\p}{\p b} F(a,b) &= \frac{\p}{\p\xi_1\dots\p\xi_k}\frac{\p}{\p b}\log\tilde{E}^b_{m,n}[e^{\sum_{l=1}^k \xi_l X_l(a)}]\Big|_{\xi_1=\cdots=\xi_k=0}\nonumber\\
%&=\frac{\p}{\p\xi_1\dots\p\xi_k}\tilde{E}^b_{m,n}[e^{\sum_{l=1}^k \xi_l X_l(a)}\overline{Y(b)}]\Big|_{\xi_1=\cdots=\xi_k=0}\\
&=\frac{\p}{\p\xi_1\dots\p\xi_{k+1}}\log\tilde{E}^b_{m,n}[e^{\sum_{l=1}^k \xi_l X_l(a)+\xi_{k+1}\ov{Y(b)}}]\Big|_{\xi_1=\cdots=\xi_k=\xi_{k+1}=0}\nonumber\\
&=\tilde{\kappa}_{k+1}^b\left(\sum_{i=1}^{t_1\wedge r}g_1(a,\tilde{R}_{i,0}^1),\sum_{i=1}^{t_1\wedge r}g_2(a,\tilde{R}_{i,0}^1),\cdots,\sum_{i=1}^{t_1\wedge r}g_k(a,\tilde{R}_{i,0}^1),Y(b) \right).\label{eqn: quenched joint cumulant 2}
\end{align}
Combining \eqref{eqn: quenched joint cumulant 1} and \eqref{eqn: quenched joint cumulant 2} gives
\[
\frac{\p}{\p a} F(a,a)=\left(\frac{\p F}{\p a} (a,b)+\frac{\p F}{\p b} (a,b)\right)\Big|_{b=a}
\]
yielding the desired result.
\end{proof}

Equation \eqref{eqn: Z deriv} and repeated application of the previous lemma now give
\begin{corollary}\label{cor: explicit deriv of log Z}
\[
\frac{\p^k}{\p a^k} \log Z_{m,n}(a)=\sum_{j=1}^k\sum_{\substack{\ell_1+\dots+\ell_j=k-j\\\ell_i\geq0}}\tilde{\kappa}^a_{j}\left(\sum_{i=1}^{t_1\wedge r}\tilde{\p}^{\ell_1}L^f(a,\tilde{R}_{i,0}^1),\cdots,\sum_{i=1}^{t_1\wedge r}\tilde{\partial}^{\ell_j}L^f(a,\tilde{R}_{i,0}^1)\right).
\]
Note that the $j=k$-th term is the $k$-th quenched cumulant of $\sum_{i=1}^{t_1\wedge r}L^f(a,\tilde{R}_{i,0}^1)$.
\end{corollary}

\subsection{Application of integration by parts}\label{sec: app-ibp}
We now use the results obtained in this section thus far to provide exact expressions for the cumulants of the partition function $\log Z_{m,n}$ (see Corollary \ref{Cor: exact formula for cumulant of free energy}).

We now impose the following assumption on $f$ to have control of the moments of $\tilde{\p}^kL^f$.
\begin{hypothesis}\label{hyp: L moments}
Assume $f$ has non-empty, open support $\supp(f)$, non-empty $D(M_f)$, $f$ is smooth on its support, and for every $k\in \Z_+$, there exist $C_k\in C(D(M_f))$, continuous functions on $D(M_f)$, such that   
\begin{equation}
    \left|\tilde{\p}^k L^f(a,x) \right| \leq C_k(a) \left(1+ |\log x|^{k+1}\right)\quad \text{ for all } (a,x)\in D(M_f)\times \supp(f)\nonumber.
\end{equation}
\end{hypothesis}

The following theorem says that all functions appearing in the four models \eqref{model-IG}, \eqref{model-G}, \eqref{model-B}, \eqref{model-IB} satisfy this hypothesis.
\begin{theorem}\label{theorem: confirming hypothesis L-moments} 
Let $f$ be one of the functions appearing in Table \ref{table f}, meaning $f$ corresponds to a gamma, inverse-gamma, beta, inverse-beta, or beta-prime distribution.  Then $f$ satisfies Hypothesis \ref{hyp: L moments}.
\end{theorem}

The proof of Theorem 6 is relegated to the appendix.

\begin{lemma}\label{lemma: hyp consequence}
Assume $f$ satisfies Hypothesis \ref{hyp: L moments}.  Let $[a_0,a_1]\subset D(M_f)$, $\eta \sim\text{uniform}(0,1)$,  and $k\in \N$.
Then the random variable
\[
\sup_{a\in [a_0,a_1]}\left|\t{\p}^kL^f(a,H^f(a,\eta))\right|
\]
has finite moments of all orders.
\begin{proof}
Put $Y:=\sup_{a\in [a_0,a_1]}\left|\t{\p}^kL^f(a,H^f(a,\eta))\right|$.  By Hypothesis \ref{hyp: L moments}, there exists a constant $C>0$ such that 
\begin{align}
    \left|\t{\p}^kL^f(a,H^f(a,\eta))\right|&\leq C \left(1 +|\log H^f(a,\eta))|^{k+1}\right)\\
    &\leq C\left(1+|\log H^f(a_0,\eta))|^{k+1}+|\log H^f(a_1,\eta))|^{k+1}\right). 
\end{align}
The last inequality follows from the monotonicity $a\mapsto H^f(a,x)$ (by equations \eqref{eqn: Lf is positive} and \eqref{eqn: H-deriv}) and holds for all $a\in [a_0,a_1]$.
Since $H^f(a_j,\eta)\sim m_f(a_j)$ for $j=0,1$, and $a_0,a_1\in D(M_f)$, both $\log H^f(a_0,\eta)$ and $\log H^f(a_1,\eta)$ have finite exponential moments. Thus $Y$ has  finite moments of all orders. 
\end{proof}

\end{lemma}
Now write $\tilde{\mathbb{P}}$, $\tilde{\E}$ for the probability measure and expectation corresponding to the environment $\tilde{\omega}$, as defined in Section \ref{sec: coupling}, and $\mathbb{P}$, $\E$ for probability measure and expectation corresponding to the environment $\omega$.  Write
\begin{align*}
\widetilde{\sigma}_0(t_1\wedge r)&:= \log Z_{m,n}(a)-\widetilde{\E}[\log Z_{m,n}(a)] \text{ and }\\
\widetilde{\sigma}_k(t_1\wedge r)&:=\frac{\p^{k}}{\p a^{k}} \log Z_{m,n}(a)\quad \text{ for } k\in \N.
\end{align*}
Similarly, define
\begin{align}
\sigma_0(t_1\wedge r)&:=\log Z_{m,n}-\E[\log Z_{m,n}] \quad \text{ and }\nonumber\\
\sigma_k(t_1\wedge r)&:= \sum_{j=1}^k\sum_{\substack{\ell_1+\dots+\ell_j=k-j\\\ell_i\geq0}}\kappa^Q_j\left(\sum_{i=1}^{t_1\wedge r}\tilde{\p}^{\ell_1}L^f(a,R_{i,0}^1),\cdots,\sum_{i=1}^{t_1\wedge r}\tilde{\partial}^{\ell_j}L^f(a,R_{i,0}^1)\right) \text{ for } k\in \N \label{def: sigma},\\
\end{align}
where $\kappa_k^Q(X_1,\dots,X_k)$ denotes the joint-cumulant of the random variable $X_1,\dots,X_k$ with respect to the quenched measure $Q_{m,n}$. 
By Corollary \ref{cor: explicit deriv of log Z},  $\sigma_k(t_1\wedge r)\overset{d}{=}\tilde{\sigma}_k^a(t_1\wedge r)$.
Recall that our environment $\tilde{\omega}$ has only changed the southern boundary random variables between the origin and the point $(r,0)$, so $\log Z_{m,n}(a)$ only has an $a$-dependence between these points.

\begin{lemma} If $f$ satisfies Hypothesis \ref{hyp: L moments} %\textcolor{red}{$R^2, Y^1,Y^2$ have finite square log-moments},
and $S_r=\sum_{i=1}^r \log R^1_{i,0}$, then for any $j,k\geq 0$,
\begin{equation}\label{eqn: ibp}
\mathbb{E}[\left(\ov{\log Z_{m,n}}\right)^j p_k(S_r,a;r)]=\sum_{\substack{\ell_1+\cdots+\ell_j=k\\ \ell_i\ge 0}}\frac{k!}{\ell_1!\cdots \ell_j!}\E\left[\prod_{i=1}^j \sigma_{\ell_i}(t_1\wedge r)\right].
\end{equation}
\end{lemma}
\begin{proof}
Write $g(a):=\E^{a}[\log Z_{m,n}]$.  Then the left-hand side of equation \eqref{eqn: ibp} is equal to
\[
\E[\left(\log Z_{m,n}-g(b)\right)^j p_k(S_r,a;r)]\big|_{b=a}.
\]
Fix $b\in D(M_f)$ and let $\mathcal{F}$ be the sigma-algebra generated by the random variables $R^1_{1,0}\dots,R^1_{r,0}$. Then there exists a measurable function function $A:\R^r\rightarrow \R$ such that $A(R^1_{1,0},\cdots,R^1_{r,0})=\E[\left(\log Z_{m,n}-g(b)\right)^j|\mathcal{F}]$ almost surely.  By Lemma A.1 from \cite{CN2018}, $A\in L^2(\mathbb{P})$.  Since $S_r\in \mathcal{F}$, Lemma \ref{lemma: higher derivatives} gives
\begin{align}
\mathbb{E}[(\log Z_{m,n}-g(b))^j p_k(S_r,a;r))]&=\mathbb{E}^a[A(R^1_{1,0},\cdots,R^1_{r,0}) p_k(S_r,a;r))]\\
&=\frac{\p^k}{\p a^k} \mathbb{E}^a[A(R^1_{1,0},\cdots,R^1_{r,0}) )]\\
&=\frac{\p^k}{\p a^k} \mathbb{E}[\left(\log Z_{m,n}-g(b)\right)^j]\label{eqn: IBP interim},
\end{align}
where $\E^a$ emphasizes that we are only taking expectations over $\{R^1_{i,0}\}_{i=1}^r$.
Now fix $a_0$ and $a_1$ such that $a\in [a_0,a_1]\subset D(M_f)$.  
Using Corollary \ref{cor: explicit deriv of log Z}, Lemma \ref{lemma: hyp consequence}, and $t_1\leq m$, we see that 
\[
\widetilde{\E}\left[\sup_{a\in[a_0,a_1]}\left|\frac{\p^k}{\p a^k} \left(\log Z_{m,n}(a)-g(b)\right)^j\right|\right]<\infty. 
\]

Thus
\begin{align*}
\eqref{eqn: IBP interim}=\frac{\p^k}{\p a^k} \tilde{\mathbb{E}}\left[\left(\log Z_{m,n}(a)-g(b)\right)^j\right]&=\tilde{\mathbb{E}}\left[\frac{\p^k}{\p a^k}\left(\log Z_{m,n}(a)-g(b)\right)^j\right]\\
&=\sum_{\substack{\ell_1+\cdots+\ell_j=k\\ \ell_i\ge 0}}\frac{k!}{\ell_1!\cdots \ell_j!}\tilde{\mathbb{E}}\left[\prod_{i=1}^j \frac{\p^{\ell_i}}{\p a^{\ell_i}}\left(\log Z_{m,n}(a)-g(b)\right)\right]
\end{align*}
Therefore
\begin{align*}
\eqref{eqn: IBP interim}\Big |_{b=a}    &=\sum_{\substack{\ell_1+\cdots+\ell_j=k\\ \ell_i\ge 0}}\frac{k!}{\ell_1!\cdots \ell_j!}\widetilde{\mathbb{E}}\left[\prod_{i=1}^j \widetilde{\sigma}_{\ell_i}(t_1\wedge r)\right]\\
&=\sum_{\substack{\ell_1+\cdots+\ell_j=k\\ \ell_i\ge 0}}\frac{k!}{\ell_1!\cdots \ell_j!}\mathbb{E}\left[\prod_{i=1}^j \sigma_{\ell_i}(t_1\wedge r)\right].
\end{align*}
\end{proof}

\begin{corollary}\label{Cor: exact formula for cumulant of free energy}
When $r=m$, and $k$ is even,
\begin{align*}
    \kappa_k(\log Z_{m,n})&=\kappa_k(E_n)-\kappa_k(S_m)-\sum_{j=1}^{k-1} \binom{k}{j}(-1)^j\kappa(\underbrace{\log Z_{m,n},\cdots,\log Z_{m,n}}_{j \text{ times}},\underbrace{S_m,\cdots,S_m}_{k-j \text{ times}})\,\,\text{ and}\\
    &=n\kappa_k(\log R^2)-m\kappa_k(\log R^1)\\
    &+\sum_{\pi\in \mathcal{P}} (|\pi|-1)!(-1)^{|\pi|}\sum_{j=1}^{k-1} \binom{k}{j}(-1)^j\prod_{B\in \pi} \mathbb{E}\left[(\overline{\log Z_{m,n}})^{a_{j,B}}p_{b_{j,B}}(S_m,a;m)\right],
\end{align*}
where $a_{j,B}=|B \cap\{1,\ldots, j\}|$, $b_{j,B}=|B \cap \{j+1,\ldots,k\}|=|B|-a_{j,B}$.
Moreover, 
\[
\mathbb{E}[(\overline{\log Z_{m,n}})^j p_k(S_m,a;m))]=\sum_{\substack{\ell_1+\cdots+\ell_j=k\\ \ell_i\ge 0}}\frac{k!}{\ell_1!\cdots \ell_j!}\E\left[\prod_{i=1}^j \sigma_{\ell_i}(t_1)\right].
\]
\end{corollary}

\section{Estimates for the central moments}\label{sec: estimates for central}

\begin{lemma}\label{lemma: sums to t}
Let $0\leq r$ and put $S_n=\sum_{i=1}^n g(a,R_{i,0})$ where $g(a,R_{i,0})$ has finite moments of all orders.  Recall
the notation \eqref{def: annealed prob}
for the annealed expectation with respect to the polymer environment. Then, for all $k\in \N$ there exist finite constants $C_k=C_k(a)>0$ which are locally bounded in $a$, such that 
\[
E_{m,n}[\left(\ov{S_{t_1}-S_{t_1\wedge r}}\right)^k]\leq C_k \left(E_{m,n}[(t_1-t_1\wedge r)^k]+1 \right) \text{ for all } (m,n)\in \N^2.
\]
Here the centering is with respect to the annealed measure $E_{m,n}$.
\end{lemma}
\begin{proof}
\begin{align}
E_{m,n}\left[\left(\ov{S_{t_1}-S_{t_1\wedge r}}\right)^k\right]&=E_{m,n}\left[\sum_{l>r}\ind_{\{t_1=l\}}\left(\ov{S_l-S_{l\wedge r}}\right)^k\right]\label{eqn: S ub 1}\\
&+(-1)^k P_{m,n}(t_1\leq r) E_{m,n}[S_{t_1}-S_{t_1\wedge r}]^k.\label{eqn: S ub 2}
\end{align}
We now treat \eqref{eqn: S ub 1} and \eqref{eqn: S ub 2} separately.
\begin{align}
&\eqref{eqn: S ub 1}= \E\left[\sum_{l>r} \left(\ov{S_l-S_{l\wedge r}}\right)^k Q_{m,n}(t_1=l)\right]\nonumber\\
&\leq \E\left[\left(\ov{S_l-S_{l\wedge r}}\right)^k\ind_{\{\ov{S_l-S_{l\wedge r}}>l-r\}}\right]\nonumber\\
&\,+\E\left[\sum_{l>r}(l-r)^k Q_{m,n}(t_1=l)\right]\nonumber\\
&\leq \E\left[\sum_{l=1}^
 \infty \left(\ov{S_l}\right)^k\ind_{\{\ov{S_l}>l\}}\right]\nonumber\\
&+ E_{m,n}\left[(t_1-t_1\wedge r)^k\right].\nonumber
\end{align}
The last inequality follows from stationarity.   Since $\ov{S_l}$ is an i.i.d.\ sum of mean zero random variables which have finite moments of all orders, 
\begin{align*}
\E\left[(\ov{S_l})^k \ind_{\{\ov{S_l}>l\}}\right]&\leq \E\left[\left(\ov{S_l}\right)^{2k}\right]^{\frac{1}{2}}\P\left(\ov{S_l}>l\right)^{\frac{1}{2}}\\
&\leq C_k l^{\frac{k}{2}}\E\left[\left(\frac{\ov{S_l}}{l}\right)^k\right]^{\frac{1}{2}}\leq C_k l^{-k}
\end{align*}
which is summable over $l$.  

For equation \eqref{eqn: S ub 2}, a slight modification of \cite[Lemma 4.2]{S} gives 
\[
\eqref{eqn: S ub 2}\leq C(E_{m,n}[t_1-t_1\wedge r]+1).
\]
Taking the $k$-th powers and using Jensen's inequality completes the proof.
\end{proof}

Given a random variable $X$ and $p\in [1,\infty)$, we write
\begin{align*}
    \|X\|_{p,\E}&:=\E[|X|^p]^{\frac{1}{p}},\\
    \|X\|_{p,E_{m,n}}&:=E_{m,n}[|X|^p]^{\frac{1}{p}}
\end{align*}
for the $p$-th norm with respect to the regular expectation $\E$ and the annealed expectation $E_{m,n}$.  When $m,n$ is understood we write $E=E_{m,n}$.  
\begin{lemma}\label{lemma: kappa norm}
For every integer $k\geq 2$ there exists a constant $C_k$ such that whenever $\{X_i\}_{i=1}^k$ are random variables with finite annealed moments, then
\[
\|\kappa^Q_k(X_1,\dots,X_k)\|_{p,\E}\leq C_k \prod_{i=1}^k \|\ov{X_i}\|_{pk,E}
\]
where the centering on the right-hand side is with respect to the annealed measure $E$.

%\[
%\E\left[|\kappa^Q_k(X_1,\dots,X_k)|^b\right]\leq C_k \prod_{i=1}^k %E_{m,n}[|\ov{X_i}|^{bk}]^{\frac{1}{k}}.
%\]
\end{lemma}
\begin{proof}
$E_{m,n}\left[X_i\right]$ are constants and therefore

\[
\kappa_k^Q(X_1,\dots,X_k)=\kappa_k^Q(\ov{X_1},\dots,\ov{X_k}).
\]

%\begin{equation}
%|\kappa_k^Q(X_1,\dots,X_k)|\leq (k-1!)\textcolor{red}{\sum_{\pi}\|\kappa_k^Q(\ov{X_1},\dots,\ov{X_k})\|_{p,\E}}.\label{eqn: kappa moment UB}
%\end{equation}
%where $\pi$ runs over all partitions of $\{1,\dots,k\}$. Using Holder's generalized inequality followed by Jensen, 
\[
\left|E^Q\left[\prod_{i\in B}| \ov{X_i} |\right]\right| ^p\leq \prod_{i\in B} E^Q\left[|\ov{X_i}|^{p |B|}\right]^{\frac{1}{|B|}}\leq \prod_{i\in B} E^Q\left[|\ov{X_i}|^{pk}\right]^{\frac{1}{k}}.
\]
Using H\"older's generalized inequality again,
\[
 \E\left[\prod_{i=1}^k E^Q[|\ov{X_i}|^{pk}]^{\frac{1}{k}}\right]\leq \prod_{i=1}^k\left(\E\left[E^Q[|\ov{X_i}|^{pk}\right]\right)^{\frac{1}{k}}
\]
Taking the $p$-th root and plugging this into \eqref{eq: joint cumulant formula 2}  yields the desired result with $C_k=(k-1)! 2^k$.
\end{proof}
The following allows us to control moments of $\sigma_k(t_1\wedge r)$ in terms of annealed moments of the exit time $t_1$. 
\begin{lemma}\label{lemma: sigma norm}
For any $k\in \N$, $p\in [1,\infty)$, there exist positive constants $C(k,p)$ such that the following two conditions hold for all $r,M\in \N$:
\begin{align}
\|\sigma_k(t_1\wedge r)\|_{p,\E}&\leq C(k,p) \left(1+\|(t_1\wedge r)^k\|_{p,E}\right)\label{eqn: sigma norm}\\
\|\sigma_k(t_1)-\sigma_k(t_1\wedge r)\|_{p,\E}&\leq C(k,p) \left(1+\|(t_1)^k\|_{2p,E}\right)\frac{\|(t_1)^M\|_{2pk,E}}{r^M}.\label{eqn: sigma norm diff}
\end{align}
\end{lemma}
\begin{proof}
For $l,m\in \N$ define 
\[
X_l(m):=\sum_{i=1}^m\widetilde{\p}^l L^f(a,R^1_{i,0}).
\]

Taking $L_p$ norms of  \eqref{def: sigma}, gives 
\begin{align}
    \|\sigma_k(t_1\wedge r)\|_{p,\E}&\leq \sum_{j=1}^k \sum_{\substack{ \ell_1+\dots \ell_j=k-j\\\ell_i\geq 0}}\|\kappa^Q_j(X_{\ell_1}(t_1\wedge r),\dots, X_{\ell_j}(t_1\wedge r))\|_{p,\E}.\label{eqn: sigma norm 2}\\
     \|\sigma_k(t_1)-\sigma_k(t_1\wedge r)\|_{p,\E}&\leq \sum_{j=1}^k \sum_{\substack{ \ell_1+\dots \ell_j=k-j\\\ell_i\geq 0}}\|\kappa^Q_j(X_{\ell_1}(t_1\wedge r),\dots, X_{\ell_j}(t_1\wedge r))\|_{p,\E}.\label{eqn: sigma norm 3}
\end{align}
Lemma \ref{lemma: kappa norm}, Lemma \ref{lemma: sums to t}, equation \eqref{eqn: sigma norm 2}, and Jensen's inequality give \eqref{eqn: sigma norm}.
By \eqref{eqn: sigma norm 3} and a telescoping argument, to obtain \eqref{eqn: sigma norm diff} it suffices to bound $\|\kappa_j^Q(Y_1,\dots,Y_j)\|_{p,E}$ where
\[
Y_i=\begin{cases}
X_i(t_1\wedge r) & \text{ for } 1\leq i<s \\
X_s(t_1)-X_s(t_1\wedge r) & \text{ for } i=s \\
X_i(t_1) & \text{ for } s<i\leq j 
\end{cases}
\]
and $s\in \{1,\cdots,j\}$ is fixed.  By Lemma \ref{lemma: sums to t} and Jensen's inequality, for $i\neq s$, 

\begin{equation}
\|\ov{Y}_i\|_{pj,\E}\leq \|\ov{Y}_i\|_{pk,\E}\leq C_{2p,k}\left( 1+\|t_1\|_{pk,E}\right)\label{eqn: sigma norm 4}.
\end{equation}
By Jensen's inequality, the Cauchy-Schwarz inequality, Lemma \ref{lemma: sums to t}, and Markov's inequality,
\begin{align*}
\|\ov{Y}_s\|_{pj,\E}&\leq    \|\ov{Y}_s\|_{pj,\E}= \|\ov{X_s(t_1)-X_s(t_1\wedge r)}\|_{pk,\E}\\
&\leq \|\ov{X_s(t_1)-X_s(t_1\wedge r)}\ind_{\{t_1>r\}}\|_{pk,\E}+\E[|X_s(t_1)-X_s(t_1\wedge r)|\ind_{\{t_1> r\}}]\\
& \leq \|\ov{X_s(t_1)-X_s(t_1\wedge r)}\ind_{\{t_1>r\}}\|_{pk,\E}+\E[|\ov{X_s(t_1)-X_s(t_1\wedge r)}|\ind_{\{t_1> r\}}]\\
&\quad +|\mathbb{E}[X_s(t_1)-X_s(t_1\wedge r)]|P_{m,n}(t_1>r)\\
&\leq 2\|\ov{X_s(t_1)-X_s(t_1\wedge r)}\|_{2pk,\E}\|\ind_{\{t_1>r\}}\|_{2pk,\E}+|\mathbb{E}[X_s(t_1)-X_s(t_1\wedge r)]|P_{m,n}(t_1>r) \\
&\leq C_{2p,k}\left(1+\|(t_1-t_1\wedge r)\|_{2pk,E}\right)P_{m,n}(t_1>r)^{\frac{1}{2pk}}\\
&\leq C_{2p,k}\left(1+\|(t_1-t_1\wedge r)\|_{2pk,E}\right)\frac{\|(t_1)^M\|_{2pk,E}}{r^M}.
\end{align*}
In the third to last inequality we again used a slight modification of \cite[Lemma 4.2]{S}.   Another application of Jensen's inequality along with \eqref{eqn: sigma norm 4} gives \eqref{eqn: sigma norm diff}.
\end{proof}

For the following lemma, recall the notation $\P^{(a_1,a_2)}$ and $\E^{(a_1,a_2)}$ defined in Section \ref{sec: four mellin}.

\begin{lemma}
Assume the polymer environment is as in \eqref{polymer environment distribution} and the sequence $(m,n)=(m_N,n_N)_{N=1}^\infty$ satisfies 
\[
|m-N\psi_1^{f^2}(a_2)|\vee|n-N\psi_1^{f^1}(a_1)|\leq \gamma N^{\frac{2}{3}}
\]
where $\gamma$ is some positive constant.  Then there exist finite positive constants $C_1,C_2,C_3,\delta,\delta_1,b$ (uniformly bounded in $(a_1,a_2)$ such that for all $N\in \N$ the following two bounds hold simultaneously for $j=1,2$:  for all $C_1 N^{\frac{2}{3}}\leq u \leq \delta N$,
\[
\P^{(a_1,a_2)}\left[Q_{m,n}(t_j\geq u)\geq e^{-\frac{\delta u^2}{N}}\right]\leq C_2 \left(\frac{N^k}{u^{2k}}\left(\E^{(a_1,a_2)}[\left(\ov{\log Z_{m,n}}\right)^k]+\E^{(a_1(\lambda_j),a_2(\lambda_j))}[\left(\ov{\log Z_{m,n}}\right)^k] \right)\right)
\]
where $a_1(\lambda):=a_1+\lambda$, $a_2(\lambda)=a_2-\lambda$, $\lambda_1=\frac{bu}{N}$, and $\lambda_2=-\frac{bu}{N}$, while for $u\geq\delta N$,
\[
\P^{(a_1,a_2)}\left[Q_{m,n}(t_j\geq u)\geq e^{-\delta_1 u}\right]\leq 2 e^{-C_3}.
\]
\end{lemma}
\begin{proof}
Follow the proof of Proposition 4.3 in \cite{CN2018} verbatim up to the displayed inequality
\[
(4.8)\leq \widehat{\P}\left[\ov{\log Z_{m,n}(a_1(\lambda_j),a_2(\lambda_j))-\log Z_{m,n}(a_1,a_2)}\geq C''' \frac{u^2}{N}\right].
\]
Now rather than bounding by the second moment, bound by the $k$-th moment to get 
\begin{align*}
(4.8)&\leq \left(\frac{N}{C'''u^2}\right)^k\widehat{\E}\left[\left(\ov{\log Z_{m,n}(a_1(\lambda_j),a_2(\lambda_j))-\log Z_{m,n}(a_1,a_2)}\right)^k\right].\\
&\leq\frac{C_2 N^k}{u^{2k}}\left(\E^{(a_1,a_2)}[\left(\ov{\log Z_{m,n}}\right)^k]+\E^{(a_1(\lambda_j),a_2(\lambda_j))}[\left(\ov{\log Z_{m,n}}\right)^k] \right).
\end{align*}
The proof of the second part is just as in Proposition 4.3 of \cite{CN2018}.
\end{proof}
\begin{corollary}\label{cor: t_1 moment}
Suppose there exist $\delta,\epsilon_0,C>0$ such that $[a_1-\epsilon_0,a_1+\epsilon_0]\times [a_2-\epsilon_0,a_2+\epsilon_0]\subset D(M_{f_1})\times D(M_{f_2})$ and the following holds for every $N\in\N$ and every $\lambda \in [-\epsilon_0,\epsilon_0]$:
\begin{equation}
\E^{(a_1(\lambda),a_2(\lambda))}\left[\left(\ov{\log Z_{m,n}}\right)^k\right]\leq C N^{(\frac{1}{3})k+\delta k}\label{eqn: assumption - moments}
\end{equation}
where $a_1(\lambda)=a_1-\lambda$, and $a_2(\lambda)=a_2+\lambda$.
Then, for all $\epsilon>0$ there exists a positive constant $C'=C'(\epsilon,k,a_1,a_2)$ such that the following bound holds for every $N\in \N$ and every $\lambda\in [-\frac{\epsilon_0}{2},\frac{\epsilon_0}{2}]$:
\[
E^{(a_1(\lambda),a_2(\lambda))}\left[(t_j)^{2k}\right]\leq C' N^{(\frac{4}{3})k+\delta k+\epsilon} \text{ for both } j=1,2.
\]
Here $E^{(a_1,a_2)}$ denotes the annealed  expectations with respect to the measure on paths in the environment \eqref{polymer environment distribution}.
\end{corollary}
\begin{proof}
Fix $\lambda_0\in [\frac{-\epsilon}{2},\frac{\epsilon}{2}]$ and put $(\tilde{a}_1,\tilde{a}_2)=(a_1(\lambda_0),a_2(\lambda_0))\in D(M_{f^1})\times D(M_{f^2})$.  Note that $\tilde{a}_1+\tilde{a}_2=a_1+a_2=a_3$ (see \eqref{polymer environment distribution}).  So by Remark \ref{remark parameter DR property} and \eqref{eqn: assumption - moments}, there exist positive constants $N_0=N_0(b,\epsilon_0)\in \N$, $C_1=C_1(\epsilon_0),\delta=\delta(\epsilon_0)$ such that for all $N\geq N_0$,
\begin{align*}
E^{(\tilde{a}_1,\tilde{a}_2)}\left[(t_j)^{2k}\right]&\leq (C_1 N^{\frac{2}{3}})^{2k} + (2k)(\epsilon N)^{\epsilon}\int_{C_1 \wedge N^{\frac{2}{3}}}^{\delta N}u^{2k-1-\epsilon} P^{(\tilde{a}_1,\tilde{a}_2)}(t_j\geq u)\mathrm{d}u+C'(\delta,\delta_1,C_3,N_0)\\
&\leq \left(C_1 N^{\frac{2}{3}}\right)^{2k} + (2k)(\delta N)^{\epsilon} C C_2 N^{(\frac{4}{3})k +\delta k} \int_{C_1 N^{\frac{2}{3}}}^{\delta N}u^{-1-\epsilon}\mathrm{d} u\\
&\leq C(\epsilon, k,\epsilon_0) N^{(\frac{4}{3}k+\delta k+\epsilon)}.
\end{align*}
Note: We needed $\epsilon$  to be large enough such that $b\delta\leq \frac{\epsilon_0}{2}$ to ensure that for all $N\in \N$ and all $u\leq \delta N$,
$\left(\tilde{a}_1\pm \frac{bu}{N}, \tilde{a}_2 \mp \frac{bu}{N}\right)\in \left\{[a_1(\lambda),a_2(\lambda)]: \lambda\in [-\epsilon_0,\epsilon_0]\right\}$.
\end{proof}
\begin{lemma}\label{lemma: truncation}
Assume the polymer environment is distributed as in \eqref{polymer environment distribution} and the sequence $(m,n)=(m_N,n_N)_{N=1}^\infty$ satisfies 
\[
|m-N\psi_1^{f^2}(a_2)|\vee|n-N\psi_1^{f^1}(a_1)|\leq \gamma N^{\frac{2}{3}}
\]
where $\gamma$ is some positive constant.  Further, suppose there exist positive constants $\delta,\epsilon_0,\{C_k\}_{k=1}^\infty$ such that $[a_1-\epsilon_0,a_1+\epsilon_0]\times [a_2-\epsilon_0,a_2+\epsilon_0]\subset D(M_{f^1})\times D(M_{f^2})$ and the following hold for every $k,N\in \N$ and every $\lambda\in [-\epsilon_0,\epsilon_0]$:

\begin{equation}
\E^{(a_1(\lambda),a_2(\lambda))}\left[\left(\ov{\log Z_{m,n}}\right)^k\right]\leq C_k N^{(\frac{1}{3}+\delta)k}\label{eqn: assumption 2}.     
\end{equation}

Then for all $\epsilon>0$, $M>0$, there exist positive constants $\{C_{j,l}=C_{j,l}(a_1,a_2,\epsilon,\delta,M)\}_{j,l=1}^\infty$  (locally bounded 
in $a_1,a_2$) such that for all $N\in \N$ we have the following:
%Then, for all $j,n\in \N$ there exist constants $C_{j,n}=C_{j,n}(a_1,a_2,\epsilon,\delta)$ (locally bounded in $(a_1,a_2)$) such that 
\[
\left|\E[\left(\overline{\log Z_{m,n}}\right)^j p_l(S_m,a_1;m)]-\E[\left(\overline{\log Z_{m,n}}\right)^j p_l(S_{\floor{\tau}},a_1;\floor{\tau})]\right|\leq C_{j,l} N^{-M},
\]
where $S_r=\sum_{i=1}^r \log R_{i,0}^1$, and $\tau=N^{(\frac{2}{3}+\frac{\delta}{2}+\epsilon)}$.
\end{lemma}
\begin{proof}
By \eqref{eqn: ibp}, 
for all $0\leq r\leq m$,
\[
\E\left[\left(\overline{\log Z_{m,n}}\right)^j p_k(S_r,a_1;r)\right]= \sum_{\substack{\ell_1+\dots+\ell_j=k\\\ell_i\geq 0}} \frac{k!}{l_1!\dots l_j!}\E\left[\prod_{i=1}^j \sigma_{\ell_i}(t_1\wedge r)\right]
\]
where $\sigma_0(t_1\wedge r)=\ov{\log Z_{m,n}}$.
It will therefore suffice to compare $\sigma_{\ell_i}(t_1)$ with $\sigma_{\ell_i}(t_1\wedge r)$.
Specifically, for fixed $\ell_1,\dots, \ell_j$, such that $\sum_{i=1}^j\ell_i=k$, we wish to estimate
\[
\E\left[\prod_{i=1}^j\sigma_{\ell_i}(t_1)-\prod_{i=1}^j\sigma_{\ell_i}(t_1\wedge r)\right].
\]

By a telescoping argument it suffices to bound
\[
\E\left[\sigma_{\ell_a}(t_1)\prod_{i\in I_1} \sigma_{\ell_i}(t_1)\prod_{i\in I_2} \sigma(t_1\wedge r)\right]-\E\left[\sigma_{\ell_a}(t_1\wedge r)\prod_{i\in I_1} \sigma_{\ell_i}(t_1)\prod_{i\in I_2} \sigma(t_1\wedge r)\right]
\]
where $a\in\{1,2,\dots,j\}$ is such that $\ell_a\neq 0$,  $I_1=\{1,\dots,a-1\}$,  $I_2=\{a+1,\dots,j\}$, and $\sum_{i=1}^j \ell_i=k$.
By the generalized H\"older inequality this is bounded by 
\begin{equation}
\|\sigma_{\ell_a}(t_1)-\sigma_{\ell_a}(t_1\wedge r)\|_{2,\E} \prod_{i\in I_1}\|\sigma_{\ell_i}(t_1)\|_{2(j-1),\E}   \prod_{i\in I_2}\|\sigma_{\ell_i}(t_1\wedge r)\|_{2(j-1),\E}   \label{eqn: telescoping}
\end{equation}
Let $I_0=\{1\leq i \leq j: \ell_i=0\}$.  By Lemma \ref{lemma: sigma norm}, for any $r,M\in \N$,  
\begin{equation}
\eqref{eqn: telescoping}\leq C(\ell_a,2)\left(1+\|(t_1)^{\ell_a}\|_{4,E}\right)\frac{\|(t_1)^M\|_{4\ell_a,E}}{r^M}\|\ov{\log Z_{m,n}}\|^{|I_0|}_{2(j-1),\E} \prod_{i\notin I_0} C(\ell_i,2(j-1))\left(1+\|t_1^{\ell_i}\|_{2(j-1),E}\right).\label{eqn: truncation 1}
\end{equation}
Using the assumption \eqref{eqn: assumption 2}, by Corollary \ref{cor: t_1 moment}, for any $\epsilon>0$ there exists a constant $C(\epsilon,p)\geq 1$ (uniformly bounded in $(a_1,a_2)$ such that
\begin{align*}
\|\ov{\log Z_{m,n}}\|_{p,\E}&\leq C(p,p) N^{(\frac{1}{3}+\delta)}\quad\quad \text{ and }\\ 
\|(t_1)^{\ell}\|_{p,E}&\leq C(\ell,p) N^{(\frac{2}{3}+\frac{\delta}{2})\ell+\frac{\epsilon}{p}}.
\end{align*}
This implies the existence of positive constants $C'=C'(k,j,M)$ such that for all $M\in \N$ and all $N\in\N$, 
\begin{equation*}
    \eqref{eqn: truncation 1}\leq C' N^{(\frac{2}{3}+\frac{\delta}{2})\ell_a+\frac{\epsilon}{4}}\cdot N^{(\frac{2}{3}+\frac{\delta}{2})M+\frac{\epsilon}{4\ell_a}}r^{-M}
    \cdot N^{(\frac{1}{3}+\delta)|I_0|}\cdot\prod_{i\notin I_0} N^{\left(\frac{2}{3}+\frac{\delta}{2}\right)\ell_i+\frac{\epsilon}{2(j-1)}}
    \end{equation*}
    Choosing 
    \[r=\floor{\tau}\ge N^{\frac{2}{3}+\frac{\delta}{2}+\epsilon}-1,\]
    we obtain the bound
    \[C' N^{(\frac{1}{3}+\delta)(j-1)+(\frac{2}{3}+\frac{\delta}{2})k +2\epsilon -M\epsilon}.\]
Now fix $M_0=M_0(\epsilon,\delta,j,k)$ large enough such that such that 
\[(\frac{1}{3}+\delta)(j-1)+(\frac{2}{3}+\frac{\delta}{2})k +2\epsilon -M_0\epsilon\leq -K.\]
\end{proof}

% \begin{corollary}\label{cor: truncation}
% With the same assumptions as in Lemma \ref{lemma: truncation}, for all $\epsilon>0$, $K>0$, and all $1\leq j\leq k$, there exist constants $C_{j,k}=C_{j,k}(\epsilon,K,a_0,a_1)$ (uniformly bounded in $a_0,a_1$) such that 
% \[
% |\kappa(\underbrace{\log Z_{m,n},\dots,\log Z_{m,n}}_{j \text{ times}},\underbrace{S,\dots,S}_{k-j \text{ times}})-\kappa(\underbrace{\log Z_{m,n},\dots,\log Z_{m,n}}_{j \text{ times}},\underbrace{S_{\floor{\tau}},\dots,S_{\floor{\tau}}}_{k-j \text{ times}})|\leq C_{j,k}N^{-K}.
% \]
% \end{corollary}
% \begin{proof}
% By Lemma \ref{lemma: joint cumulants to polynomials}, both of these cumulants are expressible as sums of products of expectations of the form $ \E\left[\left(\ov{\log Z_{m,n}}\right)^j p_k(S,a_1)\right]$.  By a telescoping argument we need only analyze expressions of the form 
% \[
%     \E[(\ov{\log Z_{m,n}})^{a_{B',l}}(p_{b_{B',l}}(S,a_1)-p_{b_{B',l}}(S_{\floor{\tau}},a_1))]\prod_{B\in \pi, B\neq B'}\E[\left(\ov{\log Z_{m,n}}\right)^{a_{B,j}}p_{b_{B,j}}(S,a_1)],
% \]
% where $a_{B,j}, b_{B,j}\ge 0$ and $a_{B,j}+b_{B,j}=|B|$, the size of block $B$ in the partition appearing in the expression for the cumulant $\kappa_k(\log Z_{m,n})$ in Corollary \ref{Cor: exact formula for cumulant of free energy}.

% The desired result now follows using assumption \eqref{eqn: assumption - moments}, Lemma \ref{lemma: truncation}, the fact that that $S$ and $S_r$ are both i.i.d. sums of random variables with finite exponential moments.
% \end{proof}

Before proceeding to Lemma \ref{improved central moment}, we note the following property of the polynomials $p_n(T,a;r)$ introduced in \eqref{eqn: recursion 2}:
\begin{proposition}\label{prop: poly-bound}
For each $n$, 
\[p_n(t,a;r)=\sum_{j} c_j(a)(t-r\psi_0(a))^{a_j}r^{b_j},\]
where $c_j(a)$ are independent of $r$ and $0\le a_j,b_j\le n$ are integers with
\begin{equation}
\label{eqn: degree-p}
    \frac{a_j}{2}+b_j=\frac{n}{2}.
\end{equation}
In particular, if $T=\sum_{k=1}^r \log X_k$ where $X_k\sim m_f(a)$, then we have for integers $b, k\ge 0$
\begin{equation}\label{eqn: poly-bound}
    \mathbb{E}[|p_b(T,a;r)|^k]\le C_{b,k}r^{kb/2}.
\end{equation}
\begin{proof}
The result is clearly true for $p_0(T,a;r)$. Next, we note that if $a_j, b_j$ satisfy \eqref{eqn: degree-p}, then
\begin{align*}
&\frac{\partial}{\partial_a} (t-r\psi_0(a))^{a_j}r^{b_j} + (t-r\psi_0(a))^{a_j}r^{b_j}\cdot (t-r\psi_0(r))\\
=&~- a_j \psi_1(a)(t-r\psi_0(a))^{a_j-1}r^{b_j+1}+(t-r\psi_0(a))^{a_j+1}r^{b_j}.
\end{align*}
Noting that 
\[\frac{a_j-1}{2}+b_j+1=\frac{a_j+1}{2}+b_j= \frac{n+1}{2},\]
the claim follows by induction from the definition \eqref{eqn: recursion 2}.
\end{proof}
\end{proposition}

\begin{lemma}\label{improved central moment}
With the same assumptions as in Lemma \ref{lemma: truncation}, for all $k\in \N$ there exist positive constants $C_k=C_k(a_1,a_2)$ (locally bounded) such that for all even $k\geq 2$.
\begin{equation}
    \E[(\ov{\log Z_{m,n}})^k]\leq C_k N^{(\frac{1}{3}+\frac{\delta}{3})k} \quad \text{ for all $N\in \N$.}\label{eqn: improved central moment}
\end{equation}
\end{lemma}
\begin{proof}
The proof is by induction on $k$. For $k=2$, \eqref{eqn: improved central moment} holds with $\delta=0$. Assuming the estimate for even exponents less than $k$, we use the first expression in Corollary \ref{Cor: exact formula for cumulant of free energy} to express the cumulant $\kappa_k(\log Z_{m,n})$ as a sum of terms of the form
\begin{equation}\label{eqn: term-to-bound}
\prod_{B\in \pi} \mathbb{E}\left[(\ov{\log Z_{m,n}})^{a_{j,B}}p_{b_{j,B}}(S_m,a_1;m)\right],
\end{equation}
where $\pi$ is a partition of $\{1,\ldots, k\}$ into $|\pi|$ blocks $B$, and $a_{j,B}+b_{j,B}=|B|$. 

Using equation \eqref{eqn: ibp} and Lemma \ref{lemma: truncation} with $K> 2k$, we have, for $\tau=n^{2/3+\delta/2+\epsilon}$,
\begin{align*}
&~\prod_{B\in \pi} \mathbb{E}\left[(\ov{\log Z_{m,n}})^{a_{j,B}}p_{b_{j,B}}(S_m,a_1;m)\right]\\
=&~\prod_{B\in \pi} \mathbb{E}\left[(\ov{\log Z_{m,n}})^{a_{j,B}}p_{b_{j,B}}(S_{\floor{\tau}},a_1;\floor{\tau})\right]+O(n^{-k}).
\end{align*}
Taking absolute values and applying H\"older's inequality,
\begin{align*}
&\left|\mathbb{E}\left[(\ov{\log Z_{m,n}})^{a_{j,B}}p_{b_{j,B}}(S_{\floor{\tau}},a_1;\floor{\tau})\right]\right|\\
\le&~\mathbb{E} [(\ov{\log Z_{m,n}})^k]^{\frac{a_{j,B}}{k}}\mathbb{E}[|p_{b_{j,B}}(S_{\floor{\tau}},a_1;\floor{\tau})|^{k'}]^{\frac{b_{j,B}}{k'}}\\
\le&~ Cn^{((1/3)+\delta/4+\epsilon/2)b_{j,B}}\mathbb{E} [(\ov{\log Z_{m,n}})^k]^{\frac{a_{j,B}}{k}},
\end{align*}
where $\frac{a_{j,B}}{k}+\frac{1}{k'}=1$. The last inequality follows from equation \eqref{eqn: poly-bound} in Proposition \ref{prop: poly-bound}.  Taking the product over $B\in \pi$, we have, up to a constant factor, the bound:
    \begin{equation}\label{eqn: apply young}n^{((1/3)+\delta/4+\epsilon/2)b_j}\mathbb{E} [(\ov{\log Z_{m,n}})^k]^{\frac{a_j}{k}},
\end{equation}
where 
\[ a_j:=\sum_B a_{j,B}\quad \text{ and } \quad b_j:=\sum_B b_{j,B},\]
so $\frac{a_j}{k}+\frac{b_j}{k}=1$.  Note that for $1\leq j\leq k-1$, we have $a_j\le k-1$.  Applying Young's inequality $xy\le \frac{1}{p}x^p+\frac{1}{q}y^q$ to \eqref{eqn: apply young}, we find that for $\eta>0$, any term of the form \eqref{eqn: term-to-bound} is bounded by
\[\eta \mathbb{E}[(\ov{\log Z_{m,n}})^k]+C(\eta)n^{((1/3)+\delta/4+\epsilon/2)k}+O(n^{-k}).\]
Combining this with Corollary \ref{Cor: exact formula for cumulant of free energy}, we have
\begin{equation}\label{eqn: eta-sum}\kappa_k(\log Z_{m,n})= C(k)\eta \mathbb{E}[(\ov{\log Z_{m,n}})^k]+C(k)C(\eta)n^{((1/3)+\delta/4+\epsilon/2)k}+O(n).
\end{equation}
Writing
\begin{equation}\label{eqn: cumulant-sum}
\kappa_k(\log Z_{m,n})=\mathbb{E}[(\ov{\log Z_{m,n}})^k]+\sum_{\substack{|\alpha|=k\\0\leq\alpha_i<k}} c_\alpha \prod_{i=1}^{|\alpha|} \mathbb{E}[(\ov{\log Z_{m,n}})^{\alpha_i}],
\end{equation}
where the sum is over multi-indices $\alpha=(\alpha_1,\ldots,\alpha_k)$, $\sum_i \alpha_i=k$. If some $\alpha_i=k-1$, then the product must equal zero.  Therefore, by the induction assumption, all terms in the sum on the right of \eqref{eqn: cumulant-sum} are of order $n^{((1/3)+\delta/3)k}$. Choosing $\eta$ sufficiently small in \eqref{eqn: eta-sum} and absorbing $\epsilon/2$ into $\delta/4$, we obtain the result.
\end{proof}

\subsection{Finishing the argument}

Combining Corollary \ref{cor: t_1 moment} and Lemma \ref{improved central moment} we obtain the following:

\begin{lemma}\label{lemma: inductive argument}
Assume the polymer environment is distributed as in \eqref{polymer environment distribution} and the sequence $(m,n)=(m_N,n_N)_{N=1}^\infty$ satisfies 
\[
|m-N\psi_1^{f^2}(a_2)|\vee|n-N\psi_1^{f^1}(a_1)|\leq \gamma N^{\frac{2}{3}}
\]
where $\gamma$ is some positive constant.  Further, suppose there exist positive constants $\delta,\epsilon_0, C(k)$ for $k\in\{2, 4, \cdots\}$  such that $[a_1-\epsilon_0,a_1+\epsilon_0]\times [a_2-\epsilon_0,a_2+\epsilon_0]\subset D(M_{f^1})\times D(M_{f^2})$ and the following hold for any even $k$ and any $\lambda\in [-\epsilon_0,\epsilon_0]$:

\begin{equation*}
\E^{(a_1(\lambda),a_2(\lambda))}\left[\left(\ov{\log Z_{m,n}}\right)^k\right]\leq C(k) N^{(\frac{1}{3}+\delta)k}.     
\end{equation*}

Then there exist constants $C'(k)>0$ for $k\in \{2,4,\dots\}$ such that for any even $k$ and any $\lambda\in [-\frac{\epsilon_0}{2},\frac{\epsilon_0}{2}]$:
\[
\E^{(a_1(\lambda),a_2(\lambda))}\left[\left(\ov{\log Z_{m,n}}\right)^k\right]\leq C(k) N^{(\frac{1}{3}+\frac{\delta}{3})k}.
\]
\end{lemma}

Theorem \ref{theorem: Main 1} will follow from repeated application of Lemma \ref{lemma: inductive argument} once we prove the following:

\begin{proposition}\label{prop: 1st step in induction}
Assume the polymer environment is distributed as in \eqref{polymer environment distribution} and the sequence $(m,n)=(m_N,n_N)_{N=1}^\infty$ satisfies 
\begin{equation}
|m-N\psi_1^{f^2}(a_2)|\vee|n-N\psi_1^{f^1}(a_1)|\leq \gamma N^{\frac{2}{3}}\label{char dir}
\end{equation}
where $\gamma$ is some positive constant. Then there exists positive constants $\epsilon_0$ and $C(k)$ for $k\in \{2,4,\dots\}$ such that $[a_1-\epsilon_0,a_1+\epsilon_0]\times [a_2-\epsilon_0,a_2+\epsilon_0]\subset D(M_{f^1})\times D(M_{f^2})$ and the following hold for any even $k$ and any $\lambda\in [-\epsilon_0,\epsilon_0]$:

\begin{equation}
    \E^{(a_1(\lambda),a_2(\lambda))}\left[\left(\ov{\log Z_{m,n}}\right)^k\right]\leq C(k) N^{(\frac{1}{3}+\frac{1}{6})k}.    
\end{equation}
\end{proposition}

\begin{proof}
Since $(a_1,a_2)\in D(M_{f^1})\times D(M_{f^2})$, there exists a positive constant $\epsilon_0$ such that  $[a_1-\epsilon_0,a_1+\epsilon_0]\times [a_2-\epsilon_0,a_2+\epsilon_0]\subset D(M_{f^1})\times D(M_{f^2})$. 
With notation as in Subsection \ref{subsec: DRP consequences}, if we define $A:=\ov{\log Z_{m,n}}$, then $A=\ov{S_m}+\ov{E_n}$.  Thus, for even $k$,
\begin{equation}
    \E[A^k]\leq 2^{k-1}(\E[(\ov{S_m})^k]+\E[(\ov{E_n})^k]).\label{eq: minkowsi bound}
\end{equation}
By Proposition \ref{proposition 4 models have DR property}, all four models described by \eqref{polymer environment distribution} have the down-right property.  So by the discussion in Subsection \ref{subsec: DRP consequences},  $S_m$ and $E_n$ are both sums of i.i.d.\  random variables whose common distributions continuously depends on $a_1$ and $a_2$ respectively.  Moreover, by Remark \ref{remark - Mellin consequences}, all random inside of the summations have finite exponential moments. Therefore, for every $k\in \{2,4,\dots\}$ there exists a positive constant $C_k=C_k(a_1,a_2)$, which is continuous in $(a_1,a_2)$, such that 
\[
\E[(\ov{S_m})^k]\leq C_km^{k/2} \quad \text{ for all } m\geq 1
\]
and
\[
\E[(\ov{E_n})^k]\leq C_kn^{k/2} \quad \text{ for all } n\geq 1.
\]
Using equations \eqref{eq: minkowsi bound} and \eqref{char dir} now yields the desired result.
\end{proof}

\begin{proof}[Proof of Theorem \ref{theorem: Main 1}]
The four basic beta-gamma models \eqref{model-IG}-\eqref{model-IB} can all be described by equation \eqref{polymer environment distribution}.  So let $\epsilon>0$ and $(a_1,a_2)\in D(M_{f^1})\times D(M_{f^2})$.  Fix even integers $k$, $M$ such that $p\leq k$ and
%and find an integer $M$ large enough such that
\[
\frac{(1/6)}{3^M}\leq \epsilon.
\]
By Jensen's inequality, it suffices to show the bounds \eqref{eqn: Main theorem 1} and \eqref{eqn: Main theorem 2} hold with $p$ replaced by $k$. Now apply Proposition \ref{prop: 1st step in induction} followed by $M$ consecutive applications of Lemma \ref{lemma: inductive argument} to obtain the bound \eqref{eqn: Main theorem 1}.  Finally, apply Corollary \ref{cor: t_1 moment} to both $t_1$ and $t_2$ to obtain the bound \eqref{eqn: Main theorem 2}.
\end{proof}

% \begin{proof} 
% The proof is by induction on $k$. When $k=2$ the result follows with $\delta=0$.  Assuming the estimate holds for all cumulants of order less than $k$, we use Corollary \ref{Cor: exact formula for cumulant of free energy} to express $\kappa_k(\log Z_{m,n})$ as a sum over joint cumulants of the form 
% \[
% \kappa_k(\underbrace{\log Z_{m,n},\dots,\log Z_{m,n}}_{j \text{ times }},\underbrace{S,\dots,S}_{j-k \text{ times}}).
% \]
% where $1\leq j\leq k-1$.
% By Lemma \ref{lemma: truncation}, it suffices to show that for fixed $1\leq j\leq k-1$,
% \[
% \kappa_k(\underbrace{\log Z_{m,n},\dots,\log Z_{m,n}}_{j \text{ times }},\underbrace{S_r,\dots,S_r}_{j-k \text{ times}})=\kappa_k(\underbrace{\log \ov{Z_{m,n}},\dots,\ov{\log Z_{m,n}}}_{j \text{ times }},\underbrace{\ov{S_r},\dots,\ov{S_r}}_{j-k \text{ times}})
% \]
% satisfies the bound \eqref{eqn: improved central moment}.
% \end{proof}

%In particular 
%\begin{align}
%\frac{\p}{\p a}B(a) &=\Cov^a\Bigl(A(\{X_k\}_{k=1}^r),S\Bigr)     \text{    and }\\
%\frac{\p^2}{\p a^2}B(a) &=\Cov^a\Bigl(A(\{X_k\}_{k=1}^r),\bigl(\overline{S}\bigr)^2\Bigr).
%\end{align}

\appendix
\section{Proof of Theorem \ref{theorem: confirming hypothesis L-moments}}
The next lemma says that it suffices to verify Hypothesis \ref{hyp: L moments} for $f(x)=e^{-bx}$, $f(x)=(1-x)^{b-1}\ind_{\{0<x<1\}}$, and $f(x)=\left(\frac{x}{1+x}\right)^b$ where $b>0$.

For $A\subset \R$ write $-A=\{ -a : a\in A\}$ and $A^{-1}=\{a^{-1}: a\in A \}$ assuming that $0\notin  A$. 

\begin{lemma}\label{lemma: f to g}
If the function $f$ satisfies Hypothesis \ref{hyp: L moments}, then so does the function $g(x):=f(\frac{1}{x})$ for $x\in (0,\infty)$, with the same constants $C_j(a)$.
\end{lemma}

\begin{proof} Recall the notation in Sections \ref{section Mellin} and \ref{sec: coupling}. 
Clearly $\supp(g)=\supp(f)^{-1}$ and $D(M_f)=-D(M_g)$. As in the proof of \cite[Lemma A.1]{CN2018}, one can verify that:
\begin{align}
F^g(a, x)&=1-F^f(-a,\frac{1}{x})\quad &\text{ for } (a,x)\in D(M_g)\times \supp(g)\nonumber\\
L^g(a,x)&=L^f(-a,\frac{1}{x})\quad &\text{ for } (a,x)\in D(M_g)\times \supp(g)\nonumber\\
H^g(a,p)&=\frac{1}{H^f(-a,1-p)} \quad &\text{ for } (a,p)\in D(M_g)\times (0,1).\label{eqn: Hg to Hf}
\end{align}
Combining the last two equalities gives 
\begin{equation}
L^g(a,H^g(a,p))=L^f(-a,H^f(-a,1-p))\quad \text{ for } (a,p)\in D(M_g)\times (0,1).\label{eqn: Lg to Lf}
\end{equation}
Recall also the definition of the derivative $\tilde{\partial}$ in \eqref{eqn: tilde-der}. We write $\tilde{\p}_f$ and $\tilde{\p}_g$ to denote the dependence on the underlying function.  Recall that
\[
\tilde{\p}_g^kL^g(a,H^g(a,p))=\frac{\p^k}{\p a^k} L^g(a,H^g(a,p))\quad \text{ for all } (a,p)\in D(M_g)\times \supp(g).
\]
Applying $\frac{\p^k}{\p a^k}$ to equation \eqref{eqn: Lg to Lf} gives
\begin{align*}
    \frac{\p^k}{\p a^k}L^g(a,H^g(a,p))&=(-1)^k\frac{\p^k}{\p b^k}\left(L^f(b,H^f(b,1-p))\right)\Big|_{b=-a}\\
    &=(-1)^k\tilde{\p}_f^kL^f(-a,H^f(-a,1-p)),
\end{align*}
so
\[
\tilde{\p}_g^kL^g(a,H^g(a,p))=(-1)^k\tilde{\p}_f^kL^f(-a,H^f(-a,1-p))\quad \text{ for all } (a,p)\in D(M_g)\times (0,1).
\]
Making the substitution $x=H^g(a,p)$ and using equation \eqref{eqn: Hg to Hf}, we get 
\[
\tilde{\p}_g^kL^g(a,x)=(-1)^k\tilde{\p}_f^kL^f(-a,\frac{1}{x})\quad \text{ for all } (a,x)\in D(M_g)\times \supp(g).
\]
Taking absolute values and using the fact that $|\log x|=|\log\frac{1}{x}|$ completes the proof.
\end{proof}

Write $C^\infty(A)$ for the set of smooth functions defined on a set $A$.  For a fixed $f$ with  non-empty $D(M_f)$ and which is smooth on its open support, define the linear transformations $T$ and $S$ on $C^\infty(D(M_f)\times \supp(f))$ by 
\begin{align*}
    T(h)(a,x)&:=\frac{1}{x^a f(x)}\int_0^x h(a,y) y^{a-1}f(y)\mathrm{d}y\\
    S(h)(a,x)&:=\frac{\p h}{\p a}(a,x)+h(a,x)\log x
\end{align*}
for $h\in C^\infty(D(M_f)\times \supp(f))$ and $(a,x)\in D(M_f)\times \supp(f)$.  Notice that when $h(a,x)=\psi^f_0(a)-\log x$, $T(h)=L^f$.  Notice that $\tilde{\p}$ in \eqref{eqn: tilde-der} is also a linear transformation on $C^\infty(D(M_f)\times \supp(f))$.  The following lemma gives a useful recursion for $\tilde{\p}^k L^f$:
\begin{lemma}\label{lemma: Lf deriv recursion}
Assume $f:(0,\infty)\rightarrow [0,\infty)$ has non-empty $D(M_f)$, open $\supp(f)$, and satisfies $f\in C^\infty(\supp(f))$. If $h\in C^\infty(D(M_f)\times \supp(f))$, then for $(a,x)\in D(M_f)\times \supp(f)$, and
\begin{align*}
(\tilde{\p}T(h))(a,x)&= T\circ S(h)(a,x) -\left[\left(a+x\frac{f'(x)}{f(x)}\right)L^f(a,x)+\log x\right]T(h)(a,x)+h(a,x)L^f(a,x).
\end{align*}
Moreover, if there exists an integer $k\geq 1$ and a constant $C=C(a_0,a_1)>0$ such that 
\begin{equation}
    \sup_{a\in[a_0,a_1]}\left|\frac{\p h}{\p a}(a,x)\right|\leq C\left(1+|\log x|^k\right)\quad \text{ for all } x\in \R,\label{eqn: Lf deriv-1}
    \end{equation}
    then 
    \[
    \int_0^\infty h(a,y)y^{a-1}f(y)\mathrm{d}y\equiv 0\Rightarrow \int_0^\infty S(h)(a,y)y^{a-1}f(y)\mathrm{d}y\equiv 0.
    \]
\end{lemma}

\begin{proof}
Computation yields
\begin{align*}
\frac{\p T(h)}{\p a}(a,x)&=-\log x\cdot T(h)(a,x) +T\circ S(h)(a,x)\\
\frac{\p T(h)}{\p x}(a,x)&=\left(-\frac{a}{x}-\frac{f'(x)}{f(x)}\right)T(h)(a,x)+\frac{h(a,x)}{x},
\end{align*}
which gives the first part.  For the second part, by Remark \ref{remark - Mellin consequences} in Section \ref{section Mellin}, $|\log X|$ has finite exponential moments.  We can therefore exchange the derivative with the integral in the expression 
\[
\frac{\p}{\p a} \int_0^\infty h(a,y) y^{a-1} f(y)\mathrm{d}y.
\]
\end{proof}
For $a\in D(M_f)$ and $x>0$, recursively define
\begin{equation}\label{eqn: hn-def}
    \begin{split}
    h_1(a,x)&:=\psi^f_0(a)-\log x\quad \text{ and }\\
    h_n(a,x)&:= S(h_{n-1})(a,x) \quad \text{ for } n\geq 2.
    \end{split}
\end{equation}
Then $h_n(a,x)$ is an $n$-th degree polynomial in $\log x$ with coefficients that are smooth in $a$.  Thus, there exist constants $C_n>0$ for $n=1,2,\dots$ such that
\[
\sup_{a\in[a_0,a_1]}\left|\frac{\p h_n}{\p a}(a,x)\right|\leq C_n\left(1+|\log x|^n\right)\quad \text{ for all } x>0.
\]
By the second part of Lemma \ref{lemma: Lf deriv recursion},
\begin{equation}
\int_0^\infty h_n(a,x)y^{a-1}f(y)\mathrm{d}y=0 \text{ for all } n\in \N \text{ and } a\in (a_0,a_1).    \label{eqn: hn integrate to 0}
\end{equation}
The functions $h_n$ will serve as a basis generating all functions obtainable from $L^f$ through repeated application of the operation $\tilde{\p}$.  To proceed we define some algebraic structures.  

Given a subset $F\subset C^\infty(D(M_f)\times \supp(f))$, define $\mathcal{A}(F)\subset C^\infty(D(M_f)\times \supp(f))$ to be the algebra generated by $F$ over the ring $C^\infty(D(M_f))$.  More specifically,
$g\in \mathcal{A}(F)\Leftrightarrow$ there exist $c_1,\dots,c_n\in C^\infty(D(M_f))$ and $g_1,\dots,g_n\in C^\infty(D(M_f)\times \supp(f))$ such that 
\[
g(a,x)=\sum_{i=1}^n c_i(a)g_i(a,x) \text{ for all } (a,x)\in D(M_f)\times \supp(f).
\]

Now let 
\begin{equation}\label{eqn: r-def}
    r(x):= x\frac{f'(x)}{f(x)} 
\end{equation}
for $x\in \supp(f)$ and put 
\[
F:=\left\{\log x,\, T(h_n)(a,x),\, r(x)T(h_n)(a,x)\,:n\in \N\right\}.
\]
Define the degree function $\mathrm{deg}:F\rightarrow \Z_+$ by
\[
\deg(\log x)=1\quad \text{ and } \mathrm{deg}(T(h_n))=\mathrm{deg}\left(r T(h_n)\right)=n \text{ for } n\in \N.
\]
Extend the degree function to $\mathcal{A}(F)$ by defining
\[
\deg(c):=0,\quad \deg(g\cdot h):=\deg(g)+\deg(h), \quad \text{ and } \deg(g+h):= \max({\deg(g),\deg(h)})
\]
for $c\in C^\infty(D(M_f))$,  and non-zero $f,g\in C^\infty(D(M_f)\times \supp(f))$.  Note that this turns $deg$ into an algebra homomorphism from $\mathcal{A}(F)$ into $\Z_+$ with the $(+,\max)$ algebra.
For $n\in \N$, let $A_n:=\{g\in \mathcal{A}(F): \deg(g)\leq n\}$.  Note that $A_n$ is linear.
\begin{lemma}\label{lemma: sufficient for L hyp}
Suppose $\supp(f)$ is non-empty and open, $D(M_f)$ is non-empty, $f$ is smooth on its support, and the following two statements hold for every $n\in N$:
\begin{enumerate}
    \item $\t{\p}r\cdot T(h_n)\in A_{n+1}$, and
    \item there exists some $C_n\in C(D(M_f))$ such that 
    \[
    \left(1\vee |r(x)|\right)|T(h_n)(a,x)|\leq C_n(a)\left (1+|\log x|^n\right) \text{ for all } (a,x)\in D(M_f)\times \supp(f).
    \]
\end{enumerate}
\end{lemma}
Then $f$ satisfies Hypothesis \ref{hyp: L moments}.
\begin{proof}
We first claim that 
\begin{equation}
\text{ The operator } \tilde{\p} \text{ maps } A_n\rightarrow A_{n+1} \text{ for all } n\in \N.\label{statement: An to An plus 1}
\end{equation}
To see this, notice that $\tilde{\p}$ satisfies a product rule:
\[
\tilde{\p}(g\cdot h)=(\tilde{\p}g)\cdot h+g\cdot (\tilde{\p}h),
\]
and it maps $C^\infty(D(M_f))\rightarrow C^\infty(D(M_f))$.  Thus, to show \eqref{statement: An to An plus 1}, it suffices to show $\tilde{\p}(\log x)\in A_1$ and for all $n\in \N$,\,\, $\tilde{\p}(T(h_n))$ and $\tilde{\p}\left(r\cdot T(h_n)\right)$ are in $A_{n+1}$. 

Clearly, $\tilde{\p}(\log x)=L^f=T(h_1)\in F$ has degree $1$ by definition, so it is in $A_1$.  By Lemma \ref{lemma: Lf deriv recursion},
\begin{align}
\tilde{\p}(T(h_n))=T(h_{n+1})-\left[(a-r(x))L^f +\log x\right]T(h_n)+h_n \cdot L^f\in A_{n+1}
\end{align}
since $T(h_{n+1})\in A_{n+1}$, $T(h_n)\in A_n$, and $r\cdot L^f, L^f, \log x\in A_1$.  Additionally, using $r\cdot T(h_{n+1})\in A_{n+1}$ and $r\cdot T(h_n)\in A_n$ we see that $r\cdot \t{\p}(T(h_n))\in A_{n+1}$ as well.  By assumption, $\t{\p}(r)\cdot T(h_n)\in A_{n+1}$, so the product rule implies 
\[
\t{\p}\left(r\cdot T(h_n)\right)\in A_{n+1},
\]
which completes \eqref{statement: An to An plus 1}.  

Now define 
\[
\mathcal{B}:=\left\{g\in \mathcal{A}(F): \text{ there exists } c\in C(D(M_f)) \text{ for which  } |g(a,x)|\leq c(a)\left(1+|\log x|^{\deg(g)}\right)\right\}.
\]
$\mathcal{B}$ is a sub-algebra of $\mathcal{A}(F)$, which clearly contains $\log x$. By assumption 2. in the statement of the Lemma,
\begin{equation}
F\subset \mathcal{B}\label{statement: F subset B},     
\end{equation}
which implies $\mathcal{B}=\mathcal{A}(F)$.  Now $L^f\in A_1$ and 
\eqref{statement: An to An plus 1} implies $(\t{\p})^n L^f\in A_{n+1}\subset \mathcal{B}$ which completes the proof.
\end{proof}

We now prove Theorem \ref{theorem: confirming hypothesis L-moments}.

\begin{proof}[Proof of Theorem \ref{theorem: confirming hypothesis L-moments}]
By Lemma \ref{lemma: f to g}, it suffices to consider only the functions $f(x)=e^{-bx}$, $f(x)=(1-x)^{b-1}\ind_{\{0<x<1\}}$, and $f(x)=\left(\frac{x}{1+x}\right)^b$ for $b>0$. We check that the assumptions of Lemma \ref{lemma: sufficient for L hyp} are satisfied in these three cases.

First note that since $h_n(a,x)$ is an $n$-th degree polynomial in $\log x$ with coefficients that are smooth in $a$, for every $n\in \N$ there exists some $\t{C}_n\in C(D(M_f))$ such that 
\begin{equation}
|h_n(a,x)|\leq \t{C}_n(a)\left(1+|\log x|^n\right) \text{ for all } (a,x)\in D(M_f)\times \supp(f).\label{eqn: hn ub}
\end{equation}
Moreover, by \eqref{eqn: hn integrate to 0}, 
\begin{equation}
\int_0^x h_n(a,y)y^{a-1}f(y)\mathrm{d}y=\int_x^\infty h_n(a,y)y^{a-1}f(y)\mathrm{d}y \text{ for all } x\geq 0.\label{eqn: h_n reverse}
\end{equation}
\underline{Case 1}: $f(x)=e^{-bx}$.   Here $\supp(f)=(0,\infty)=D(M_f)$, $f$ is clearly smooth on $(0,\infty)$, and $r(x)=-bx$. Notice that 
\[
\tilde{\p}r=-bxL^f=r\cdot L^f
\]
lies in $A_1$ by definition.  So clearly
the first assumption of Lemma \ref{lemma: sufficient for L hyp} holds. 
 We now check the second assumption of Lemma \ref{lemma: sufficient for L hyp} is satisfied.  When $0<x\leq 1$, using \eqref{eqn: hn ub} and $a>0$,
\begin{align*}
|T(h_j)(a,x)|\leq \frac{e^{bx}}{x^a}\int_0^x|h_n(a,y)|y^{a-1}e^{-y}\mathrm{d}y &\leq \frac{e^b \t{C}_j(a)}{x^a}\int_0^x\left(1+|\log y|^j\right)y^{a-1}\mathrm{d}y\\
&\leq C_j(a)\left(1+|\log x|^j\right).
\end{align*}

When $x>1$, using \eqref{eqn: hn ub}, \eqref{eqn: h_n reverse}, followed by the substitution $y\mapsto \frac{y}{x}-1$, and finally an application of integration by parts yields

\begin{align*}
|T(h_j)(a,x)| &\leq\frac{e^{bx} \t{C}_j(a)}{x^a}\int_x^\infty\left(1+|\log y|^j\right)y^{a-1}e^{-by}\mathrm{d}y\\
&= \t{C}_j(a)\int_0^\infty \left(1+|\log(y+1)+\log x|^j\right)y^{a-1}e^{-bxy}\mathrm{d}y.\\
&=\frac{\t{C}_j(a)}{b x}\left(1+|\log 2|^j+|\log x|^j \right) +O(\frac{1}{x^2})\\
&\leq \frac{\tilde{C}_j(a)}{bx}\left(1+|\log x|^j\right),
\end{align*}
where we increased $\t{C}_j(a)$ in the last step if necessary.  Combining these two bounds yields the desired result, completing the proof for Case 1. 

\underline{Case 2}: $f(x)=(1-x)^{b-1}\ind_{\{0<x<1\}}$.   Here $\supp(f)=(0,1)$, $D(M_f)=(0,\infty)$, $f$ is clearly smooth on $(0,1)$, and $r(x)=-(b-1)\frac{x}{1-x}$.  To see 
that the first assumption in Lemma \ref{lemma: sufficient for L hyp} holds, notice that 
\[
\tilde{\p}r=-(b-1)\frac{x}{(1-x)^2}L^f=r\cdot \left( 1+\frac{r}{1-b}\right)L^f.
\]
Thus $\t{\p}r \cdot T(h_j)=(r\cdot L^f)\cdot T(h_j)+\frac{1}{b-1} (r\cdot L^f)\cdot (r\cdot T(h_j))\in A_{j+1}$ since $r\cdot L^f \in A_1$, and $r\cdot T(h_j)\in A_j$ by definition.  We now check the second assumption of Lemma \ref{lemma: sufficient for L hyp} is satisfied.  By \eqref{eqn: hn ub}, we have the bounds
\[
\left|h_j(a,y)y^{a-1}f(y)\right|\leq {\small \begin{cases}
\t{C}_j(a)\big(1+|\log y |^j\big)y^{a-1} &\mbox{if } 0<y<\frac{1}{2}\\
\t{C}_j(a)(1-y)^{b-1} &\mbox{if } \frac{1}{2}\leq y <1.
\end{cases}}
\]

Since $a>0$, for $0 <x< \frac{1}{2}$,
\begin{equation}
\left|T(h_j)(a,x)\right|\leq \frac{2^a C_j(a)}{x^a}\int_0^x(1+|\log y|^j)y^{a-1}\mathrm{d}y\leq \tilde{C}_j(a)\left(1+|\log x|^j\right)\label{T bound beta case 1}    
\end{equation}
Similarly, using equation \eqref{eqn: h_n reverse}, for $\frac{1}{2}\leq x<1$,
\begin{equation}
    \left|T(h_j)(a,x)\right|\leq \frac{2^a \t{C}_j(a)}{(1-x)^{b-1}}\int_x^1 (1-y)^{b-1} \mathrm{d}y \leq \t{C}_j(a) (1-x) \label{eqn: T bound beta case 2}.\\
\end{equation}
where we increased $\t{C}_j(a)$ if necessary.
Thus, for all $0<x<1$, $|r(x)|\leq |b-1|\frac{1}{1-x}$ implies
\[
\left|T(h_j)(a,x)\right|\vee \left|r(x) T(h_j)(a,x)\right|\leq C_j(a)\left(1+|\log x|^j\right).
\]
This completes the proof for case 2.

\underline{Case 3}: $f(x)=\left(\frac{x}{1+x}\right)^b$.   Here $\supp(f)=(0,\infty)$, $D(M_f)=(-b,0)$, $f$ is clearly smooth on $(0,\infty)$, and $r(x)=b\frac{1}{1+x}$.  To see the first assumption of Lemma \ref{lemma: sufficient for L hyp} is satisfied,  notice that 
\[
\tilde{\p}r=-b\frac{x}{(1+x)^2}L^f=-r\cdot \left(1-\frac{r}{b}\right)L^f.
\]
Thus $\t{\p}r\cdot T(h_j)=-(r\cdot L^f)\cdot T(h_j)+\frac{1}{b} (r\cdot L^f)\cdot (r\cdot T(h_j))\in A_{j+1}$ since $r\cdot L^f \in A_1$, and $r\cdot T(h_j)\in A_j$ by definition.  We now check the second assumption of Lemma \ref{lemma: sufficient for L hyp}.  By \eqref{eqn: hn ub}, we have the bounds
\[
\left|h_j(a,y)y^{a-1}f(y)\right|\leq {\small \begin{cases}
\t{C}_j(a)\big(1+|\log y |^j\big)y^{a+b-1} &\mbox{if } 0<y<1\\
\t{C}_j(a)y^{a-1} &\mbox{if } 1\le  y <\infty .
\end{cases}}
\]

Since $a+b>0$, for $0 <x< 1$,
\begin{equation}
\left|T(h_j)(a,x)\right|\leq \frac{2^b C_j(a)}{x^{a+b}}\int_0^x(1+|\log y|^j)y^{a+b-1}\mathrm{d}y\leq \tilde{C}_j(a)\left(1+|\log x|^j\right)\label{T bound beta case 3}    
\end{equation}
Similarly, using equation \eqref{eqn: h_n reverse}, for $1\leq x<\infty $,
\begin{equation}
    \left|T(h_j)(a,x)\right|\leq \frac{2^b \t{C}_j(a)}{x^a}\int_x^\infty y^{a-1} \mathrm{d}y \leq \tilde{C}_j(a) \label{eqn: T bound beta case 4}.\\
\end{equation}
where we increased $\tilde{C}_j(a)$ if necessary.
Thus, for all $0<x<\infty$, $|r(x)|\leq |b|$ implies
\[
\left|T(h_j)(a,x)\right|\vee \left|r(x) T(h_j)(a,x)\right|\leq C_j(a)\left(1+|\log x|^j\right).
\]
This completes the proof for case 3.

\end{proof}

% \section{Some computations}
% Look at the integral
% \[\int_x^\infty (\psi_0(\theta)-\log y) y^{\theta-1} e^{-y}\,\mathrm{d}y.\]
% We estimate this. First we treat
% \[x^{-\theta}e^x\int_x^\infty e^{-y}y^{\theta-1}\mathrm{d}y.\]
% Shifting the integral, we have
% \[x^{-\theta}e^x\int_x^\infty e^{-y}y^{\theta-1}\mathrm{d}y= x^{-\theta}\int_0^{\infty} e^{-y}(x+y)^{\theta-1}\mathrm{d}y.\]
% Rescaling $y\mapsto yx$, we have
% \[x^{-\theta}\int_0^{\infty} e^{-y}(x+y)^{\theta-1}\mathrm{d}y=\int_0^\infty e^{-xy} (1+ y)^{\theta-1}\,\mathrm{d}y.\]
% By integration by parts, the last integral is of order $1/x$:
% \[\int_0^\infty e^{-xy} (1+ y)^{\theta-1}\,\mathrm{d}y=\frac{1}{x}+\frac{\theta-1}{x}\int_0^\infty e^{-xy}(1+y)^{\theta-2}\,\mathrm{d}y=\frac{1}{x}+O(1/x^2).\]
% For the integral
% \[x^{-\theta}e^x \int_x^\infty y^{\theta-1} \log y\  e^{-y}\,\mathrm{d}y, \]
% we have a similar treatment. First shifting by $x$, we have
% \[x^{-\theta}\int_0^\infty e^{-y} \log(x+y) (x+y)^{\theta-1}\mathrm{d}y\]
% Resacaling, we have
% \begin{align*}\int_0^\infty (\log x+\log(1+y))(1+y)^{\theta-1} e^{-xy}\,\mathrm{d}y=&\log x \int_0^\infty (1+y)^{\theta-1}e^{-xy}\,\mathrm{d}y\\
% &+ \int_0^\infty (1+y)^{\theta-1} \log (1+y) e^{-xy}\,\mathrm{d}y.
% \end{align*}
% The first term is bounded by
% \[\frac{\log x}{x}.\]
% The second integral can also be treated by integration by parts:
% \[\int_0^\infty (1+y)^{\theta-1} \log (1+y) e^{-xy}\,\mathrm{d}y=\int_0^\infty e^{-xy} \theta(1+y)^{\theta-2}\,\mathrm{d}y\]

\section{Finite exponential moments for the free energy}
\begin{lemma}\label{lemma: fin exp. moments for log Z}
Assume the polymer environment is such that $|\log R^1|$, $|\log R^2|$, $|\log Y^1|$, and $|\log Y^2|$ all have finite exponential moments.   Then,  
\[
\left|\log Z_{m,n}\right| \text{ has finite exponential moments for all } (m,n)\in Z_+^2.
\]
\end{lemma}
\begin{proof}
Since $\log Z_{0,0}=0$, $\log Z_{k,0}=\sum_{i=1}^k R^1_{i,0}$, and $\log Z_{0,k}= \sum_{j=1}^k R^2_{0,j}$, $\log Z_x$ has finite exponential moments for any $x\in Z_+^2\setminus \N^2$.  When $x\in \N^2$, the recursion \eqref{eqn: recursion 2} implies that 
\[
\left(\log Y_x^1 +\log Z_{x-\alpha_1}\right)\wedge \left(\log Y_x^2 +\log Z_{x-\alpha_2}\right)\leq \log Z_x-\log 2\leq \left(\log Y_x^1 +\log Z_{x-\alpha_1}\right)\vee \left(\log Y_x^2 +\log Z_{x-\alpha_2}\right).
\]
Thus 
\[
|\log Z_x -\log 2|\leq \left|\log Y_x^1 +\log Z_{x-\alpha_1}\right|\vee \left|\log Y_x^2 +\log Z_{x-\alpha_2}\right|.
\]
Since $|\log Y^1_x|$ and $|\log Y^2_x|$ have finite exponential moments, and inductive argument finishes the proof.
\end{proof}

%\bibliography{ref}

\end{document}